\documentclass[12pt, a4paper]{amsart}

\usepackage{amssymb, amsmath}
\usepackage[all]{xy}
\usepackage[inline]{enumitem}
\usepackage{hyphenat}
\usepackage[colorlinks,linkcolor=blue,citecolor=blue,urlcolor=teal]{hyperref}
\usepackage{xcolor}
\usepackage{mathtools}
\usepackage{tikz-cd}
\usepackage[labelformat=empty]{caption}
 \usepackage[ left=3cm, right=3cm, top = 2.8cm, bottom = 2.8cm ]{geometry}

\makeatletter
\def\@tocline#1#2#3#4#5#6#7{\relax
  \ifnum #1>\c@tocdepth % then omit
  \else
    \par \addpenalty\@secpenalty\addvspace{#2}%
    \begingroup \hyphenpenalty\@M
    \@ifempty{#4}{%
      \@tempdima\csname r@tocindent\number#1\endcsname\relax
    }{%
      \@tempdima#4\relax
    }%
    \parindent\z@ \leftskip#3\relax \advance\leftskip\@tempdima\relax
    \rightskip\@pnumwidth plus4em \parfillskip-\@pnumwidth
    #5\leavevmode\hskip-\@tempdima
      \ifcase #1
      \or\or \hskip 2em \or \hskip 2em \else \hskip 3em \fi%
      #6\nobreak\relax
    \dotfill\hbox to\@pnumwidth{\@tocpagenum{#7}}\par
    \nobreak
    \endgroup
  \fi}
\makeatother

\makeatletter
\newcommand{\myitem}[1]{%
\item[#1]\protected@edef\@currentlabel{#1}%
}
\makeatother

\newcommand{\eq}[2]{\begin{equation}\label{#1}#2 \end{equation}}
\newcommand{\ml}[2]{\begin{multline}\label{#1}#2 \end{multline}}
\newcommand{\mlnl}[1]{\begin{multline*}#1 \end{multline*}}

\newcommand{\arir}{\ar@{^{(}->}}
\newcommand{\aril}{\ar@{_{(}->}}
\newcommand{\are}{\ar@{>>}}

\newcommand{\xr}[1] {\xrightarrow{#1}}

\newcommand{\lra}{\longrightarrow}
\newtheorem{lem}{Lemma}[section]
\newtheorem{thm}[lem]{Theorem}

\newtheorem{prop}[lem]{Proposition}

\newtheorem{cor}[lem]{Corollary}

\theoremstyle{definition}
\newtheorem{defn}[lem]{Definition}
\newtheorem{defn-prop}[lem]{Definition-Proposition}

\newtheorem{nota}[lem]{Notation}
\newtheorem{para}[lem]{}

\newtheorem*{acknowledgement}{Acknowledgement}

\theoremstyle{remark}

\newtheorem{rmk}[lem]{Remark}

\newtheorem{exs}[lem]{Examples}
\newtheorem{exs-rmks}[lem]{Examples and Remarks}

\newtheorem{claim}{Claim}[lem]
\newtheorem*{claim*}{Claim}

\newcounter{zaehler} 
\numberwithin{equation}{lem}

\newcommand{\N}{\mathbb{N}}

\newcommand{\Q}{\mathbb{Q}}
\newcommand{\Z}{\mathbb{Z}}

\renewcommand{\P}{\mathbf{P}}
\newcommand{\A}{\mathbf{A}}

\newcommand{\G}{\mathbf{G}}

\newcommand{\sL}{\mathcal{L}}
\newcommand{\sO}{\mathcal{O}}

\newcommand{\sU}{\mathcal{U}}
\newcommand{\sV}{\mathcal{V}}
\newcommand{\sX}{\mathcal{X}}
\newcommand{\sY}{\mathcal{Y}}
\newcommand{\sZ}{\mathcal{Z}}

\newcommand{\fm}{\mathfrak{m}}

\newcommand{\fp}{\mathfrak{p}}

\newcommand{\Xb}{{\overline{X}}}

\newcommand{\Db}{{\overline{D}}}

\newcommand{\tF}{{\widetilde{F}}}

\newcommand{\tX}{{\widetilde{X}}}

\newcommand{\ux}{{\underline{x}}}
\newcommand{\uy}{{\underline{y}}}

\newcommand{\Cor}{\operatorname{\mathbf{Cor}}}

\newcommand{\HI}{\operatorname{\mathbf{HI}}}

\newcommand{\RSC}{{\operatorname{\mathbf{RSC}}}}
\newcommand{\RSCNis}{{\operatorname{\mathbf{RSC}}}_{\Nis}}

\newcommand{\ul}[1]{{\underline{#1}}}

\newcommand{\PST}{{\operatorname{\mathbf{PST}}}}

\newcommand{\NST}{\operatorname{\mathbf{NST}}}

\newcommand{\Hom}{\operatorname{Hom}}
\newcommand{\uHom}{\operatorname{\underline{Hom}}}

\newcommand{\Ker}{\operatorname{Ker}}

\renewcommand{\Im}{\operatorname{Im}}
\newcommand{\Coker}{\operatorname{Coker}}

\newcommand{\Div}{\operatorname{Div}}

\newcommand{\Spec}{\operatorname{Spec}}

\newcommand{\Sw}{\operatorname{Sw}}
\newcommand{\Art}{\operatorname{Art}}

\newcommand{\td}{\operatorname{trdeg}}

\newcommand{\ab}{{\rm ab}}
\renewcommand{\sp}{{\rm sp}}
\newcommand{\dlog}{\operatorname{dlog}}

\newcommand{\Sm}{\operatorname{\mathbf{Sm}}}

\newcommand{\tr}{{\operatorname{tr}}}
\newcommand{\Ztr}{{\operatorname{\mathbb{Z}_{\tr}}}}

\newcommand{\red}{{\operatorname{red}}}

\newcommand{\Nis}{{\operatorname{Nis}}}

\newcommand{\et}{{\operatorname{\acute{e}t}}}

\newcommand{\inj}{\hookrightarrow}

\newcommand{\surj}{\rightarrow\!\!\!\!\!\rightarrow}

\newcommand{\id}{{\operatorname{id}}}

\newcommand{\codim}{{\operatorname{codim}}}
\newcommand{\ch}{{\operatorname{ch}}}

\newcommand{\CH}{{\operatorname{CH}}}
\newcommand{\Frac}{{\operatorname{Frac}}}

\newcommand{\e}{{\epsilon}}
\renewcommand{\b}{{\rm b}}
\renewcommand{\c}{{\rm c}}
\newcommand{\mc}{{\rm mc}}

\newcommand{\gen}{{\rm gen}}

\newcommand{\colim}{\operatornamewithlimits{\varinjlim}}

\newcommand{\ol}{\overline}

\renewcommand{\epsilon}{\varepsilon}
\renewcommand{\div}{\operatorname{div}}

\renewcommand{\sp}{{\rm sp}}
\newcommand{\la}{\langle}
\newcommand{\ra}{\rangle}

\newcommand{\MNST}{\operatorname{\mathbf{MNST}}}

\newcommand{\MCor}{\operatorname{\mathbf{MCor}}}

\newcommand{\MPST}{\operatorname{\mathbf{MPST}}}
\newcommand{\CI}{\operatorname{\mathbf{CI}}}

\newcommand{\CItspNis}{\CI^{\tau,sp}_{\Nis}}

\newcommand{\bcube}{{\ol{\square}}}

\newcommand{\ulMPST}{\operatorname{\mathbf{\underline{M}PST}}}

\newcommand{\uMPST}{\operatorname{\mathbf{\underline{M}PST}}}
\newcommand{\uMNST}{\operatorname{\mathbf{\underline{M}NST}}}

\newcommand{\uMCor}{\operatorname{\mathbf{\underline{M}Cor}}}

\newcommand{\uomega}{\underline{\omega}}
\newcommand{\ulomegaCI}{\underline{\omega}^{\CI}}

\def\rmapo#1{\overset{#1}{\longrightarrow}}

\def\PXD{P^{(D)}_X}

\def\Fh{F^h}
\def\sOh{\sO^h}

\def\isom{\overset{\sim}{\longrightarrow}}
\def\Art{\mathrm{Art}}

\def\Sw{\mathrm{Sw}}

\def\rmapo#1{\overset{#1}{\longrightarrow}}
\def\PXD{P^{(D)}_X}

\title[Zariski-Nagta purity for reciprocity sheaves]{Ramification theory of reciprocity sheaves, I\\ Zariski-Nagata purity}
\author{Kay R\"ulling \and Shuji Saito}
 
\address{Bergische Universit\"at Wuppertal\\ Gau\ss str. 20, D-42119 Wuppertal, Germany}
\email{ruelling@uni-wuppertal.de}
 
\address{Graduate School of Mathematical Sciences, University of Tokyo, 3-8-1 Komaba, Tokyo 153-8941, Japan}
\email{sshuji@msb.biglobe.ne.jp}

\thanks{K.R.\ was supported by the DFG Heisenberg Grant RU 1412/2-2. 
S.S.\ is supported by the JSPS KAKENHI Grant (20H01791). }
\begin{document}
\begin{abstract}
We prove a Zariski-Nagata purity theorem for the motivic ramification filtration of a reciprocity sheaf.
An important tool in the proof is a generalization of the Kato-Saito reciprocity map from geometric global
class field theory to all reciprocity sheaves. As a corollary we obtain cut-by-curves and cut-by-surfaces
criteria for various ramification filtrations. In some cases this reproves known theorems, 
in some cases we obtain new results.
\end{abstract}
\maketitle

\tableofcontents

\def\Keta{K_\eta}

\section{Introduction}
In this paper we prove a Zariski-Nagata purity theorem for the motivic ramification filtration of a reciprocity sheaf,
using a generalization of the Kato-Saito reciprocity map \cite{KS-GCFT} to all reciprocity sheaves.
As a corollary we obtain cut-by-curves and cut-by-surfaces criteria for various ramification filtrations. 
\para\label{H1}
Let $X$ be a smooth scheme over a perfect field $k$ and let $D$ be an effective Cartier divisor on $X$ such that its support 
has simple normal crossings and denote by  $U=X\setminus D$ its complement.
%Let $X^{(i)}$ for $i\geq 0$ denotes the set of points of codimension $i$ in $X$.
Denote by
\eq{intro;H1}{
H^1(U):=H^1(U_{\et},\Q/\Z)=\Hom_{\rm cts}(\pi_1^{\ab}(U),\Q/\Z),}
the group of torsion characters of the  abelianized fundamental group $\pi_1^{\ab}(U)$ of $U$.
There are  three different ways to express that  a character $\chi\in H^1(U)$ has
ramification bounded  by $D$:
\begin{enumerate}[label=(\roman*)]
\item\label{intro:cut-by-curve1} 
For any generic point $\eta\in D$, we have $\Art_{\Keta}(\chi_{|\Keta})\leq v_\eta(D)$,
where $\Keta=\Frac(\sO_{X,\eta}^h)$ is the quotient field of the henselization of $X$ at $\eta$, 
$v_\eta(D)$ is the multiplicity of $\eta$ in $D$,
and $\Art_{\Keta}$ is the Kato-Matsuda Artin conductor for characters of the Galois group of 
$\Keta$. (In the notation of \cite[Definition 3.2.5]{Matsuda}, $\Art_{\Keta}=1+\Sw'_{\Keta}$).
\item\label{intro:cut-by-curve2} 
For any $k$-morphism $h:C\to X$ with $C$ a smooth curve, such that $h^{-1}(U)$ is dense in $C$, and any
closed point $x\in h^{-1}(D)$, we have
$\Art_{L_x}(\chi_{|L_x})\leq v_x(h^*D)$,
where $L_x=\Frac(\sO_{C,x}^h)$ and $v_x(h^*D)$ is the multiplicity of $x$ in the pullback $h^*D$.
\footnote{Note that $\Art_{L_x}$ is the classical Artin conductor since the residue field of $L_x$ is perfect.}
\item\label{intro:cut-by-curve3} 
Let $\PXD$ 
%$= \Bl_{\Delta_X\cap (D\times D)}(X\times X) \backslash \big(\widetilde{X\times D}\cup \widetilde{D\times X}\big)$
be the blow-up of $X\times X$ in the center $\Delta_X\cap (D\times D)$ with the strict transforms of 
$X\times D$ and $D\times X$ removed, where $\Delta_X$ denotes the diagonal.
% and denote by
%$p_1,p_2: \PXD \rightrightarrows X$
%the two maps induced by the two projections. 
Note that we have a natural open immersion $U\times U\hookrightarrow \PXD$. 
Then, 
\[p_1^*(\chi)-p_2^*(\chi)\in \Im\big(H^1(\PXD) \inj  H^1(U\times U)\big),\]
where $p_i:U\times U\to U$ denotes the projection to the $i$th factor.
\end{enumerate}

It is known that the above conditions are equivalent:
the equivalence of (i) and (ii) follows from the ramification theory developed in \cite{KatoSwan} and \cite{Matsuda} (see \cite[Corollary 2.8]{Kerz-Saito}).
The condition  \ref{intro:cut-by-curve3} is introduced in Abbes-Saito's non-logarithmic ramification theory (\cite[Section 2, Subsection 3.1]{TSaito}) and
the equivalence of (i) and (iii) follows from \cite[Proposition 8.8]{AS}, \cite[Proposition 2.27]{TSaito} and \cite[Theorem 0.1]{yat}.

%A purpose of this paper is to generalize the above equivalence to a context, which we propose to call 
%\emph{motivic ramification theory} 

One aim of this  paper (and its sequel \cite{RS-AS}) is to  generalize the 
statements above to all reciprocity sheaves and prove that they are still equivalent, see Theorem \ref{intro:thm1} below.

\para\label{reviewRSC}
We recall the notion of reciprocity sheaves introduced by Kahn, Saito and Yamazaki,
which is closely related to the theory of modulus sheaves with transfers developed by the same authors in 
joint work with Miyazaki, see \cite{KSY1}, \cite{KSY2}, \cite{KMSY1}, and \cite{KMSY2}.
Fix a perfect field $k$.  Let $\Cor$ be the  category whose objects are
the smooth separated $k$-schemes and finite correspondences as morphisms. 
% (see \ref{para:Cor} for a precise definition). 
Let $\PST$ be the category of presheaves with transfers, i.e.,  the presheaves of abelian groups on $\Cor$.
%The  $\A^1$-invariant presheaves with transfers form a full subcategory of $\PST$  and 
%are at the heart of Voevodsky's theory of motives.
A reciprocity presheaf $F$ is such an object of $\PST$ that any section
$a\in F(U)$ has 
{\em a modulus}: this condition is a generalization of Voevodsky's $\A^1$-invariance and is better understood using
the theory of modulus presheaves with transfers, which we now recall. A {\em modulus pair} $(X, D)$ 
consists of a separated scheme $X$ of finite type over $k$ and a (possibly empty) effective Cartier divisor $D$ on $X$ 
such that the complement $U=X\setminus D$ is smooth. The category of modulus pairs has as
morphisms the finite correspondences between the smooth complements satisfying a certain admissibility condition 
with respect to the divisors. A presheaf on the category of modulus pairs  is called a 
{\em modulus presheaf with transfers}, we denote by  $\uMPST$ the corresponding category. 
If  the modulus pair $(X,D)$ is proper, i.e., $X$ is proper over $k$, then we can associate to it a
particular presheaf with transfers, which is denoted by $h_0(X,D)$. It satisfies the following two properties, where we write $Z_K= Z\otimes_k K$ for a $k$-scheme $Z$ 
and a function field $K$ over $k$ and where we extend a presheaf with transfers to essentially smooth $k$-schemes   by taking colimits, see \eqref{para:RSC1.6}:
\begin{enumerate}[label=(\alph*)]
\item\label{intro:h01} $h_0(X,D) $ is a quotient of the representable presheaf with transfers $\Ztr(U)=\Cor(-,U)$, where $U=X\backslash D$, and 
\item\label{intro:h02} for any function field $K$ over $k$, 
the map $\Ztr(U)(K)\to h_0(X,D)(K)$ from \ref{intro:h01} induces an isomorphism
\[\CH_0(X_K, D_K)\xr{\simeq}h_0(X,D)(K) ,\]
via the identification of $\Ztr(U)(K)$ with $Z_0(U_K)$, the group of zero-cycles on $U_K$.
Here $\CH_0(X_K|D_K)$ is the Chow group of zero-cycles with modulus introduced in 
\cite{Kerz-Saito}, which is the quotient of $Z_0(U_K)$ by the subgroup generated by 
$\div_C(f)$, where $C$ is an integral curve in $U_K$ and $f$ is a
rational function on the normalization $\tilde{C}$ of the closure of $C$ in $X_K$ which satisfies 
$f\equiv 1 \text{ mod } D_{\widetilde{C}}$, i.e., at the points of $D_{\widetilde{C}}$ the function $f-1$ is  regular and 
contained in the ideal sheaf of $D_{\widetilde{C}}$, where $D_{\widetilde{C}}=D_K\times_{X_K} \widetilde{C}$.
\end{enumerate}
In fact the Zariski sheafification of $h_0(X,D)$ is uniquely determined by \ref{intro:h01} and \ref{intro:h02}.\footnote{This follows from the results in \cite{KSY2}  together with \cite[Theorem 3.3]{BS19}.}
A key role is played by the functor
\[
\ulomegaCI: \PST  \longrightarrow \ulMPST,
\]
given by
\[ \ulomegaCI F(X,D) := \colim_{(Y,E)}\Hom_{\PST}(h_0(Y,E),F),\quad 
\text{for } F\in \PST,
\]
where $(X,D)$ is a modulus pair and the colimit is  over the cofiltered ordered set of compactifications $(Y,E)$ of $(X,D)$, see \cite[1.8]{KMSY1}.
Note that by \ref{intro:h01} we have $\ulomegaCI F(X,D)\subset F(U)$.
A modulus for $a\in F(U)$ is a proper modulus pair $(X,D)$ as above such that $a$ is contained in 
$\ulomegaCI F(X,D)$. We arrive at the following definition. 

\begin{defn}\label{def;RSCNis}
Let $F\in \PST$ and  write $\tF:=\ulomegaCI F\in \ulMPST$.
Following \cite{KSY2}, we say that $F$ is a \emph{reciprocity sheaf} if for any smooth $k$-scheme $U$
the restriction $F_{U}$  is a sheaf on the small Nisnevich site of $U$ and any section 
$a\in F(U)$ has a modulus, i.e.,
\eq{def;RSCNis1}{F(U)=\bigcup_{(X,D)} \tF(X,D),}
where the union is over all proper modulus pairs $(X,D)$ with $U=X\setminus D$.
We  denote by $\RSCNis\subset\PST$ the full subcategory of reciprocity sheaves.

We say that a reciprocity sheaf $F$ has {\em level $n\ge 0$}, 
if for any smooth $k$-scheme $X$ and any $a\in F(\A^1\times X)$ the following implication holds:
\[a_{\A^1_z}\in F(z)\subset F(\A^1_z), \quad \text{for all } z\in X_{(\le n-1)}  \Longrightarrow a\in F(X)\subset F(\A^1\times X),\]
where $a_{\A^1_z}$ denotes the restriction of $a$ to $\A^1_z=\A^1\times z\subset \A^1\times X$,
$X_{(\le n-1)}$ denotes the set of points in $X$ whose closure has dimension $\le n-1$, and for a smooth scheme $S$ we identify
$F(S)$ with its image in $F(\A^1\times S)$ via pullback along the projection map.
\footnote{This is equivalent to  the motivic conductor of $F$ having level $n$ in the language of \cite{RS}.}
\end{defn}

Any $\A^1$-invariant Nisnevich sheaf with transfers is a reciprocity sheaf of level $0$.
A relevant reciprocity sheaf to the introduction is $F=H^1$ from \ref{H1} and it has level $1$. 
See \cite[Part 2]{RS}, \cite[\S11.1]{BRS}, and \ref{exs:ZNP} for other examples.

Heuristically, a presheaf with transfers $F$ satisfies \eqref{def;RSCNis1} if  any section $a\in F(U)$ has ``bounded ramification along boundaries of compactifications of $U$" and moreover $F$ is $\A^1$-invariant if any $a\in F(U)$ has ``tame ramification at boundaries".
A manifestation of this viewpoint has been given in \cite{RS}, 
where \emph{the motivic conductor} associated to a reciprocity sheaf $F$ is introduced.
It is a collection of maps 
\[c^F=\{c^F_L : F(L) \to \N\}_{L}\]
where $L$ runs through all  henselian discrete valuation fields of geometric type over $k$,
defined for $a\in F(L)$ by
\[ c_L^F(a) =\min\{n\in \N\mid a\in \tF(\Spec \sO_L,n\cdot s_L)\}, \]
where $s_L\in \Spec \sO_L$ is the closed point.
It is shown in \cite{RS} that $c^F$ recovers classically known conductors:
for $F=H^1$ in \ref{H1}, $c^F_L=\Art_L$ from (i);
for $F=  {\rm Conn}^1$, the group of rank $1$ connections, $c^F_L$ is the irregularity (up to a shift);
if $F$ is represented by a commutative algebraic group and $\td(L/k)=1$, $c^F_L$ is Rosenlicht-Serre's conductor.
This is why  for $F\in \RSCNis$ we call
\eq{intro;tF}{ \tF(X,D)\subset F(U),\quad  \text{with }(X,D)
\text{ proper modulus pairs with } X\setminus D=U,}
the \emph{motivic ramification filtration on $F(U)$}.
\medbreak

We say that \emph{resolutions of singularities hold over $k$ in dimension $\le n$} if for any integral projective $k$-scheme
$Z$ of dimension $\le n$ and any effective Cartier divisor $E$ on $Z$, there exists a proper birational morphism $h:Z'\to Z$
such that $Z'$ is regular and $|h^{-1}(E)|$ has simple normal crossings. 
This is known to hold if ${\rm char}(k)=0$ by Hironaka or if
$n\leq 3$ by \cite{Cossart-Piltant}.
The following is a consequence of the main result of this paper, see  Corollary \ref{cor:cut-by-curve}.

\begin{thm}\label{intro:thm1}
Let $\Xb$ be a smooth projective $k$-scheme and $\Db$ be an effective Cartier divisor on $\Xb$ whose support has
simple normal crossings. Let $X\subset \Xb$ be a non-empty open subscheme, $D=\Db\cap X$ and $U=X\setminus |D|$.
Let $F$ be a reciprocity sheaf of  level $n\ge 0$.  
Assume resolutions of singularities hold over $k$ in dimension $\le n$. For $\chi\in F(U)$, the following statements are equivalent:
\footnote{
In case $X=\ol{X}$, the implications $(i) \Leftrightarrow (iii) \Leftrightarrow (iv) \Rightarrow (ii)$ hold without assuming resolution of singularities in dimension $\le n$.}
\begin{enumerate}
\myitem{(i)}\label{intro:thm1i}
For any generic point $\eta\in D$, we have $c^F_{K_\eta}(\chi_{|\Keta})\leq v_\eta(D)$
(cf. \ref{H1}\ref{intro:cut-by-curve1}).
\myitem{(ii)}\label{intro:thm1ii}
For any $k$-morphism 
$h:Z\to X$ with $Z$ smooth quasi-projective of dimension $\le n$ such that the support of $D_Z=h^* D$ 
has simple normal crossings and for any generic point $x\in D_Z$, we have 
$c^F_{L_x}(\chi_{|L_x})\leq v_x(D_Z)$, where $L_x=\Frac(\sO^h_{Z,x})$ (cf. \ref{H1}\ref{intro:cut-by-curve2}).
\myitem{(iv)}\label{intro:thm1iv} $\chi\in \tF(X,D)$.
\end{enumerate}
\end{thm}

Building on the present article and \cite{RS-HLS},  we will show in  \cite{RS-AS} the equivalence of \ref{intro:thm1iv} and the following statement:
\begin{enumerate}
\myitem{(iii)} \label{intro:thm1iii}
Let  $\PXD$ and $p_1,p_2: U\times U\rightrightarrows U$ be as in \ref{H1}\ref{intro:cut-by-curve3}. Then,
\[p_1^*(\chi)-p_2^*(\chi)\in \Im\Big(F(\PXD) \rmapo{j^*}   F(U\times U)\Big).\]
\end{enumerate}
Note that the equivalence of \ref{intro:thm1i} and \ref{intro:thm1iv} proven in the present paper reads as the equality
\eq{intro:gen}{\tF(X,D)=F_\gen(X,D).}
Here by definition
\[F_\gen(X,D):=\Ker\left(F(U)\to \bigoplus_\eta
\frac{F(\Spec \sO_{X,\eta}^h\backslash\eta)}{\tF(\Spec \sO_{X,\eta}^h, D^h_\eta)}\right),\]
where  the direct sum is over the generic points $\eta$ of $D$ and  $D^h_\eta=D\times_X \Spec \sO_{X,\eta}^h$.
This can be  viewed as a Zariski-Nagata purity for the motivic ramification filtration.
If the support of $D$ is smooth, this also follows from \cite[Corollary 8.6(2)]{S-purity} and if $D$ is reduced, it follows
from \cite[Theorem 2.3]{SaitologRSC}. The general case treated here is new.
On the other hand, the equivalence of \ref{intro:thm1iv}  and \ref{intro:thm1ii}  in the case $n=1$ is viewed as a cut-by-curves criterion for 
the motivic ramification filtration.
See \ref{exs:ZNP} for examples where we obtain unconditional results, in particular the cut-by-surfaces criterion for the ramification of fppf-torsors under 
infinitesimal finite unipotent group schemes over $k$ is new. 
\def\XK{X_K}
\def\qz{\Q/\Z}

\para
It is usually very difficult to compute Chow groups with modulus used in the definition of the motivic ramification filtration (see the description of $h_0(X,D)$ especially
\ref{reviewRSC}\ref{intro:h02} above) so that it seems hopeless to compute the filtration plainly from its definition. On the other hand, condition \ref{intro:thm1iii} 
from above, which is considered in \cite{RS-AS}, is easy to check and it provides an effective method to compute the filtration.
In the proof of Theorem \ref{intro:thm1}, we use  another effective method for the computation, called \emph{reciprocity pairing}, 
see \eqref{intro;ZNP-pairing} and Theorem \ref{intro;ZNP} below.

Let $K$ be a function field over $k$ and let $X$ be a reduced and projective  $K$-scheme of pure dimension $d$.
We denote by $K^M_{r, X}$ ($r \ge 0$) the Nisnevich sheafification 
of the improved Milnor K-theory from \cite{Kerz}. For  a nowhere dense closed subscheme $D\subset X$, we consider the following Nisnevich sheaves for $r\ge 1$\footnote{See \ref{para:KM} for a comparison of $V_{r,X|D}$ with the sheaf $K^M_{r}(\sO_X, I_D)=\Ker(K^M_{r, X}\to K^M_{r, D})$ used in \cite[(1.3)]{KS-GCFT}, where $I_D\subset \sO_X$ is the ideal sheaf for $D$.}
\[V_{r,X|D}:=\Im(\sO_{X|D}^\times\otimes_{\Z}K^M_{r-1, X}\to K^M_{r,X})
\;\;\text{ with } \sO_{X|D}^\times:= \Ker(\sO_X^\times\to i_*\sO_D^\times).\]
%Note $V_{1, X|D}=\sO_{X|D}^\times$.
Let $U\subset X$ be a regular dense open subscheme such that $U\cap D=\emptyset$.
For  a closed point $x\in U$, the Gersten resolution (\cite[Proposition 10(8)]{Kerz})  yields an isomorphism 
\[\theta_x: \Z\xr{\simeq} H^d_x(U_{\Nis}, K^M_{d, U})
\cong H^d_x(X_{\Nis}, V_{d, X|D}),\]
which induces a surjective map (cf. \cite[Theorem 2.5]{KS-GCFT}): 
\eq{thetaU}{ \theta_U=\sum_x \theta_x: Z_0(U)=\bigoplus_{x\in U} \Z
\surj H^d(X_{\Nis}, V_{d, X|D}),}
where $x$ runs through all closed points of $U$. Consider the pairing
\eq{intro;pairing}{
(-,-)_{U\subset X/K}:F(U) \otimes Z_0(U) \to F(K),\quad a\otimes  [x] \mapsto  (g_x)_*i_x^*(a),}
where $i_x:x\to U$ is the closed immersion and $(g_x)_*:F(x) \to F(K)$ is the transfer map for the finite map $g_x :x\to \Spec K$ induced by $X\to \Spec K$.
A key result used in
the proof of Theorem \ref{intro:thm1} is Proposition \ref{prop:mod-K-pairing}, stating  that 
\eqref{intro;pairing} induces a pairing
\eq{intro;ZNP-pairing}{(-,-)_{(X,D)/K}:F_\gen(X,D)\otimes H^d(X_\Nis, V_{d, X|D})\to F(K),}
where $F_\gen(X,D)$ is defined similarly as in \eqref{intro:gen}, in particular it depends only on the 
points of codimension at most $1$ in the normalization of $X$.
 
In Examples and Remarks \ref{rmk:CFT}, it is explained that the pairing \eqref{intro;pairing} recovers some classically known pairings:
in case $k$ is finite and $F(X)=H^1(X)$ from \ref{H1}, it recovers the reciprocity map constructed in \cite[(3.7)]{KS-GCFT}. In case ${\rm char}(k)=0$ and $F(X)$ is the group of isomorphism classes of absolute  rank one connections on $X$ relative to $k$,  this pairing 
is a higher-dimensional version of  the pairing constructed  in \cite[4.]{Bloch-Esnault} 
for $X$  a curve.
See also \ref{rmk:CFT}\ref{rmk:CFT3} for a pairing involving the non-$\A^1$-invariant part of \'etale motivic cohomology with mod-$p^n$ coefficients 
induced by the pairing \eqref{intro;ZNP-pairing}, which is reminiscent of the pairings constructed in \cite{JSZ} and \cite{Gupta-KrishnaII}, though in different cohomological degrees.

\begin{thm}[Theorem \ref{thm:ZNP}]\label{intro;ZNP}
Let $F$ be a reciprocity sheaf. Let $X$ be a smooth projective $k$-scheme of pure dimension $d$ and $D$ an effective Cartier divisor on $X$ whose 
support has simple normal crossings. Denote by $j: U=X\setminus D\inj X$ the open immersion.
For $a\in F(U)$, the following conditions are equivalent:
\begin{enumerate}[label=(\roman*)]
\item\label{intro:thm:ZNP1}
$a\in \tF(X,D)$;
\item\label{intro:thm:ZNP2}
$a\in F_{\gen}(X,D)$;
\item\label{intro:thm:ZNP3}
for any function field $K$ over $k$, the map
\[(a_K,-)_{U_K\subset X_K/K} : Z_0(U_K) \to F(K)\]
induced by the pairing \eqref{intro;pairing} factors through $H^d(X_{K,\Nis}, V_{d, X_K|D_K})$, where $a_K\in F(U_K)$ with $U_K=U\otimes_k K$ is the  pullback of $a$.
\end{enumerate}
\end{thm}

Note that the assumption on the smoothness of $X$ in the above theorem cannot be relaxed to just requiring the complement 
$X\setminus D$ to be smooth.
Indeed, in Theorem \ref{cor:rat-sing} we show as an application of the above that a normal Cohen-Macaulay scheme $Y$ which is of finite type over a field of characteristic zero has rational singularities if and only 
if there exists locally on $Y$ 
an effective Cartier divisor $E$  whose support contains the singular locus of $Y$, such that
the natural map $\widetilde{\Omega^d}(Y,E)\to \Omega^d_{\rm gen}(Y,E)$ is bijective. 
\medbreak

We sketch the strategy of the proof of Theorem \ref{intro;ZNP}, which implies  Theorem \ref{intro:thm1}. 
%The implication of $\ref{intro:thm:ZNP3} \Rightarrow \ref{intro:thm:ZNP1}$ in Theorem \ref{intro;ZNP} is a formal consequence of the existence of the cycle map constructed in \cite{RSCycleMap}, cf. Proposition \ref{prop:cyle-map}.
%\shuji{
The implication $\ref{intro:thm:ZNP1} \Rightarrow \ref{intro:thm:ZNP2}$ is direct from the definition. 
The implication of $\ref{intro:thm:ZNP3} \Rightarrow \ref{intro:thm:ZNP1}$ is 
a formal consequence of \ref{reviewRSC}\ref{intro:h02} above, \cite[Theorem 3.1]{S-purity} and a commutative diagram
\[\xymatrix{
& \ar[ld]_{\pi} Z_0(U_K) \ar[rd]^{\theta_{U_K}} \\
\CH_0(X_K|D_K) \ar[rr]^-{cyc_{X_K|D_K}} && H^d(X_{K,\Nis}, V_{d, X_K|D_K})\\}\]
where $\pi$ is the quotient map, $\theta_{U_K}$ is the map \eqref{thetaU} and $cyc_{X_K|D_K}$ is a cycle class map constructed in \cite{RSCycleMap}, cf. Proposition \ref{prop:cyle-map}.

Thus the most essential point is the implication $\ref{intro:thm:ZNP2} \Rightarrow \ref{intro:thm:ZNP3}$
in Theorem \ref{intro;ZNP}.
This requires to construct the pairing \eqref{intro;ZNP-pairing}. 
To this end we extend the formalism of pushforward for cohomology of reciprocity sheaves from \cite{BRS} to that for some ad hoc version of compactly 
supported cohomology, which is done in section \ref{sec:pfcs}. 
Using this pushforward, we construct an approximation of the desired pairing, see
\eqref{para:pairing2} and Lemma \ref{lem:pair-cor}. 
% \ref{para:pairing}. 
To show that this pairing induces \eqref{intro;ZNP-pairing}, we use a description of  the cohomology group $H^d(X_\Nis, V_{d, X|D})$ in terms of the sections
of $V_{d, X|D}$ over henselizations of $X$ along Par\v{s}in chains from \cite{KS-GCFT}, which is recalled in section \ref{sec:parsin}. 
This reduces the problem to a purely local situation in which case it follows from
Theorem \ref{thm:cont}, which relies on a weak version of the Brylinski-Kato formula for reciprocity sheaves, see Corollary \ref{cor:BK} and
see Remark \ref{rmk:BK} for an explanation why this is called (weak) Brylinski-Kato formula\footnote{The full Brylinski-Kato formula for reciprocity sheaves 
will be  proven in \cite{RS-AS}}.  

\begin{acknowledgement}
The second named author thanks Ahmed Abbes, Takeshi Saito and Yuri Yatagawa for their kindly  answering questions on ramification theory. 
The authors thank the referee for his comments.
\end{acknowledgement}

\begin{nota}\label{nota}
For a noetherian ring $R$, ${\rm Frac}(R)$ denotes its total ring of fractions. 

In the following $k$ denotes a field and $\Sm$ the category of separated schemes 
which are smooth and of finite type over $k$. For $k$-schemes $X$ and $Y$ we write $X\times Y:= X\times_k Y$.
For $n\ge 0$ we write $\P^n:=\P^n_k$, $\A^n:=\A^n_k$.
We say a field $L$ is a {\em henselian dvr of geometric type over $k$}, if there exists $U\in \Sm$ and  $x\in U$ a 
1-codimensional point, such that $L={\rm Frac}(\sO_{U,x}^h)$; we denote the ring of integers of $L$ by $\sO_L$.

For a scheme $X$ we denote by $X_{(i)}$ (resp. $X^{(i)}$) 
the set of $i$-dimensional (resp. $i$-codimensional) points of $X$.
If $(R,\fm)$ is a local ring, then we denote by 
\[R\{x_1,\ldots, x_n\}\]
the henselization of  the polynomial ring $R[x_1,\ldots, x_n]$ at the ideal 
$\fm R[x_1,\ldots, x_n]+ (x_1,\ldots, x_n)$.

Let $F$ be a Nisnevich sheaf on a scheme $X$ and $x\in X$ a point.
Then we denote by $F_x$ its Zariski stalk and by $F_x^h=\varinjlim_{x\in U/X}F(U)$ the Nisnevich stalk, where the colimit
is over all Nisnevich neighborhoods $U\to X$ of $x$.

\end{nota}

\section{Preliminaries}
In this section $k$ is a perfect field.

\begin{para}\label{para:mod}
We recall some basic notions from \cite{KMSY1} and \cite{KMSY2}.
A {\em modulus pair} is a pair $\sX=(X,D)$, where $X$ is a separated scheme of  finite type over $k$, $D$ 
is an effective Cartier divisor on $X$ ($D=\emptyset$ is allowed) and $X\setminus|D|$ is smooth;  $\sX$ is said to be proper
if $X$ is. For modulus pairs $\sX=(X,D)$ and $\sY=(Y,E)$ we denote by $\uMCor(\sX, \sY)$ 
the free abelian group generated by
integral closed subschemes $V\subset X\setminus|D|\times Y\setminus|E|$  
which are finite and surjective over a connected component of $X\setminus|D|$ and such  that 
$\ol{V}^N\to X\times Y$ the  normalization of the closure of $V$ is proper over $X$  and satisfies
\eq{para:mod1}{ D_{|\ol{V}^N}\ge E_{|\ol{V}^N}.} 
The category of modulus pairs with morphisms $\uMCor(\sX, \sY)$ and composition induced by the 
composition of finite correspondences is denoted by $\uMCor$; its
full subcategory of proper modulus pairs by $\MCor$.
We have a fully faithful functor $\Cor\to \uMCor$, $X\mapsto (X,\emptyset)$,  and we will abbreviate notation and write
\eq{para:mod2}{(X,\emptyset)=:X \quad \text{in }\uMCor.}
Furthermore $\uMPST$ (resp. $\uMNST$) denotes the category of presheaves (resp. Nisnevich sheaves) on $\uMCor$,
similarly with $\MPST$ and $\MNST$, for details see {\em loc. cit.}

There is a symmetric monoidal structure on $\uMCor$ defined by
\eq{para:mod3}{(X,D)\otimes (Y,E)= (X\times Y, p_X^* D+ p_Y^*D),}
where $p_X$ and $p_Y$ define the projections to $X$ and $Y$, respectively.
It extends to a symmetric monoidal structure $\otimes_{\uMPST}$ on $\uMPST$.
\end{para}

%In \ref{para:RSC} and \ref{para:twist} below we assume $k$ to be {\em perfect}.

\begin{para}\label{para:RSC}
We denote by $\PST$ (resp. $\NST$) the category of presheaves (resp. Nisnevich sheaves) with transfers on $\Sm$.
Let  $F\in \PST$. Following \cite[Definition 2.2.4]{KSY2} we say $a\in F(U)$ has {\em modulus} $\sX=(X,D)$
(called SC-modulus in {\em loc. cit.}), if
$\sX$ is a proper modulus pair with $U=X\setminus|D|$ and the Yoneda map $a:\Ztr(U):=\Cor(-,U)\to F$ 
factors via  the quotient $q: \Ztr(U)\surj h_0(\sX)$, where $h_0(\sX)$ is the presheaf with transfers given on $S\in \Sm$ by
\eq{para:RSC1}{h_0(\sX)(S):=
\Coker\left(\uMCor(\bcube\otimes S, \sX)\xr{i_0^*-i_1^*} \Ztr(U)(S)\right),\quad \bcube=(\P^1, \infty),}
with $i_e: \{0\}\to \P^1$ the inclusion of the $e$-section, $e\in \{0,1\}$.

We say $F$ is a reciprocity  presheaf if any $a\in F(U)$ has a modulus,  for any $U\in\Sm$.
The category of reciprocity presheaves is denoted by $\RSC$ and 
$\RSC_\Nis=\RSC\cap \NST$ denotes  the category of reciprocity sheaves; both are abelian, 
for the latter see \cite[Theorem 0.1]{S-purity}. 

For $F\in \RSC_\Nis$ and $\sX=(X,D)$ we denote by $\tF(X,D)$ those sections $a\in F(X\setminus|D|)$
having a modulus of the form $(\ol{X}, \ol{D}+ B)$, where $\ol{D}$ and $B$ are effective Cartier divisors such that 
$X=\ol{X}\setminus|B|$ and $\ol{D}_{|X}=D$.
Then $\sX\mapsto \tF(\sX)$ defines a Nisnevich sheaf on $\uMCor$ and with  the notation from 
\cite[2.4]{KMSY1} and \cite[Proposition 2.3.7]{KSY2} we have
\eq{para:RSC1.5}{\tF=\tau_!\omega^{\CI} F,}
which is also equal to $\uomega^{\CI} F$ from the introduction.
Moreover $\tF$ is $\bcube$-invariant, semi-pure, and has $M$-reciprocity, i.e.,
\[\tF\in \CItspNis\subset \uMNST,\]
see e.g. \cite[1.]{MS} for notation and the references there. 
We also recall that for any $G\in \CItspNis$, we have $\uomega_!G\in \RSC_\Nis$ and $\uomega_!\tF=F$, where
\[\uomega_!G(X)=G(X,\emptyset).\]
Finally if $(X,D)=\varprojlim_i (X_i, D_i)$ is a pro-modulus pair in the sense of \cite[3.7]{RS}
we define 
\eq{para:RSC1.6}{G(X,D)=\varinjlim_i G(X_i,D_i).}
If $R$ is a regular noetherian $k$-algebra and $I\subset R$ is an ideal which is invertible   defining
a Cartier Divisor $D_I$ on $\Spec R$, then $(\Spec R, D_I)$ is a pro-modulus pair (by a result of Popescu \cite{Popescu})
 and we set
\eq{para:RSC2}{ G(R, I^{-1}):= G(\Spec R, D_I).}
Usually we will refer to a pro-modulus pair simply as modulus pair.
\end{para}

\begin{para}\label{para:twist}
We recall some twists introduced in \cite[5e]{RSY} and \cite[2.]{MS}.
By \cite[Corollary 4.18]{RSY} there is a lax monoidal structure on $\RSC_\Nis$
which in particular assigns to $F_1,F_2\in \RSC_\Nis$ a new reciprocity sheaf  denoted by
\[(F_1,F_2)_{\RSC_{\Nis}}\]
(denoted by $h_{0,\Nis}(\tF_1\otimes_{\uMPST} \tF_2)$ in {\em loc. cit.}).
Let $F\in \RSC_\Nis$. For $n\ge 0$ we define
\eq{para:twist1}{ F\la 0\ra := F, \quad F\la n+1\ra:= (F\la n\ra, \G_m)_{\RSC_\Nis}, }
and 
\eq{para:twist2}{\gamma^0 F:= F, \quad \gamma^1 F:=\uHom_{\PST}(\G_m, F), \quad 
\gamma^{n+1}F:=\gamma^1 (\gamma^n F). }
For $F\in \RSC_\Nis$ and $n,m\ge 0$ we have
\begin{enumerate}[label=(\arabic*)]
\item\label{para:twist3.1} $\gamma^n F$, $F\la n\ra\in \RSC_\Nis$ and 
$F\la n+m\ra= F\la m\ra\la n\ra$, $\gamma^{m+n} F= \gamma^n(\gamma^m F)$;
\item\label{para:twist3.2} $F\in \HI_\Nis$ (=$\A^1$-invariant sheaves in $\NST$)
                 $\Rightarrow$  $F\la n\ra = F\otimes_{\HI} \G_m^{\otimes_{\HI} n}$ (see \cite[Theorem 5.3]{RSY});
\item\label{para:twist3.3} 
there is a natural surjection $F\otimes_{\NST} K^M_n\surj F\la n\ra$ in $\NST$ and 
an isomorphism $\gamma^n F=\uHom_{\PST}(K^M_n, F)$, see \cite[Proposition 9.3]{BRS},
where $K^M_n$ denotes the improved $n$-th Milnor $K$-sheaf from \cite{Kerz}.
\item\label{para:twist3.4} Furthermore we will use the following shorthand  in the rest of the text
\[\gamma^n F\la m\ra:= \gamma^n(F\la m\ra)=\uHom_{\PST}(K^M_n, F\la m\ra).\]
The weak cancellation theorem in the form \cite[Theorem 5.2]{MS} yields
\[\gamma^n F\la m\ra \cong 
\begin{cases} 
F\la m-n\ra, &  m\ge n,\\ 
\gamma^{m-n} F, & n\ge m.
\end{cases}\]
\end{enumerate}
\end{para}

\section{A Brylinski-Kato formula and continuity of a pairing}
In this section $k$ is a perfect field. We fix $G\in \CItspNis$ and we will use the shorthand notation
\eq{notation2}{G(X):=\uomega_!G(X):=G(X,\emptyset).}

\begin{prop}\label{prop:BK}
Let $(R, \fm)$ be a  regular henselian local $k$-algebra 
of geometric type and set $\sX=\Spec R\{t\}$ (see Notation \ref{nota}),
 $\sZ=V(t)\subset \sX$.
For $a\in R\{t\}$ and $n\ge 2$ denote by $f_a: \sX\to \sX$ the morphism induced
by the unique local $R$-algebra morphism $R\{t\}\to R\{t\}$ with $t\mapsto t+at^n$. 
Then $f_a$ defines a morphism of modulus pairs $f_a: (\sX, n\sZ)\to (\sX,n\sZ)$  and 
the map induced by pullback along $f_a$
\[f_a^*: \frac{G(\sX,n\sZ)}{G(\sX)}\lra \frac{G(\sX,n\sZ)}{G(\sX)}\]
is equal to the identity.
\end{prop}

We first need the following lemma which is a version of \cite[Lemma 6.7]{S-purity}.

\begin{lem}\label{lem:gc-rel-curve}
Let $R$ be an excellent normal henselian local ring with infinite residue field, set $S=\Spec R$, 
and denote by $s\in S$ its closed point.
Let $\pi: X\to S$ be a smooth morphism  of relative dimension one and let $Z\subset X$ be a 
closed irreducible subscheme which is finite over $S$ and  such that we have 
an isomorphism of residue fields $\kappa(s)\xr{\simeq}\kappa(x)$, where 
$x\in Z$ denotes the closed point.
Then 
\[\sO_{X,x}^h=\varinjlim_{(U,x_U)\to (X,x)} \sO(U),\]
where the colimit is over all Nisnevich neighborhoods $u:(U, x_U)\to (X,x)$
satisfying the following properties
\begin{enumerate}[label=(\arabic*)]
\item\label{lem:gc-rel-curve1} $U$ is affine;
\item\label{lem:gc-rel-curve2} $u$ induces an isomorphism $Z_U:=u^{-1}(Z)\xr{\simeq} Z$;
\item\label{lem:gc-rel-curve3} $(\pi\circ u: U\to S, Z_U)$  has a good compactification in the sense of 
\cite[Definition 2.1]{S-purity}.(Recall that this means that $\pi\circ u$ factors as  $U\xr{j} \ol{U}\xr{\ol{\pi}_U}S$
with $j$ an open immersion, $\ol{\pi}_U$ proper, $\ol{U}$ is normal, $\ol{U}\setminus U$ is the support of an effective
Cartier divisor $U_\infty$, and $Z_U\sqcup U_\infty\subset$ affine open of $\ol{U}$.)
\end{enumerate}
\end{lem}
\begin{proof}
First we show that $(X,x)$ admits one Nisnevich neighborhood with the stated property.
Note that by assumption $Z=\Spec R'$, with $R'$ henselian local.
By \cite[Theorem 10.0.1]{Levine} (see also the explanations in \cite[5.1]{Levine})
we find a Nisnevich neighborhood $v:(V, x_V)\to (X, x)$
which admits a closed immersion $i: V\inj \A^n_S$, such that
its closure $\ol{V}\subset \P^n_S$ admits a finite and surjective map to $\P^1_S$ and such that
$\ol{V}\setminus V$ is finite over $S$.
In particular, 
$\ol{V}\setminus V= (\P^n_S\setminus\A^n_S)\cap \ol{V}$ is a Cartier divisor in $\ol{V}$
and $\ol{V}$ is equidimensional of relative dimension one over $S$.
Denote by $Y\to \ol{V}$ the normalization; it is 
an isomorphism over $V$, moreover it is a finite morphism, since $S$ is excellent.
Since $Z$ is henselian local, we can write $v^{-1}(Z)= Z_V\sqcup C$ with 
$Z_V\cong Z$ and the components of $C$ are either finite over $Z$ or don't meet the special fiber.
Since $\ol{V}\setminus V$ is finite over $S$, the intersection of the closure $\ol{C}\subset Y$ with the special fiber $Y_s$ is finite;
hence by \cite[Lemma 12.1]{S-purity} $\ol{C}$ is finite over $S$.
Since $Z_V$ is finite over $S$, we have $\ol{C}\cap Z_V=\emptyset$.
By \cite[Theorem 5.1]{GLL} we find an effective ample Cartier divisor $H\subset Y$
such that $\ol{C}\subset H$, $H\cap (Z_V\cap Y_s)=\emptyset$, and $H$ does not contain $Y_s$. 
By \cite[Lemma 12.1]{S-purity}, we have $H\cap Z_V=\emptyset$ and $H$ is finite over $S$.
Similarly, we find an effective ample Cartier divisor $A\subset Y$ such that
$H\cup (Y\setminus V)\cup Z_V\subset Y\setminus A$; note that $Y\setminus A$ is affine.
We set $U:= V\setminus (V\cap H)$ and $x_U:=x_V$. Then $(U, x_U)\to (X,x)$ is a 
Nisnevich neighborhood which satisfies \ref{lem:gc-rel-curve1} - \ref{lem:gc-rel-curve3},
where we can take  $U\inj Y$ as a good compactification of $(U\to S, Z_U)$.

If $\nu:(X', x')\to (X,x)$ is any Nisnevich neighborhood, we find a $Z'\subset \nu^{-1}(Z)$ mapping isomorphically 
to $Z$. Now replacing $(X,Z,x)$ in the above discussion by $(X', Z', x')$ shows
 that any Nisnevich neighborhood of $(X,x)$ has a refinement satisfying \ref{lem:gc-rel-curve1} - \ref{lem:gc-rel-curve3}.
\end{proof}

Recall from \cite[Definition 5.1]{S-purity} that a pair $(X\to S,Z)$ as in Lemma \ref{lem:gc-rel-curve}
is by definition a {\em nice V-pair} if  it admits a good compactification, $Z$ is reduced and \'etale over $S$, and
$nZ$ diagonally embedded into $nZ\times_S X$ is a principal divisor, for all $n\ge 1$.

\begin{proof}[Proof of Proposition {\ref{prop:BK}}.]
%By \cite[Thm 1.8]{Popescu} and the functorial properties of the henselization, $R$ is a filtered inductive limit of henselian local $k$-algebras which are essentially smooth over $k$. Thus we may assume $R=(B_\fp)^h$, where $B$ is a smooth $k$-algebra and $\fp\in \Spec B$, in particular $R$ is excellent, regular, and henselian.
%In case  $k$ is finite,  denote by $k_\ell$ the maximal prime-to-$\ell$ extension of $k$, where $\ell$ is a prime number. Let $\fp_{\ell}$ be a prime ideal in $B\otimes_k k_{\ell}$ over $\fp$ and consider $R_\ell=((B\otimes_k k_{\ell})_{\fp_\ell})^h$. Note that $R_\ell$ is an excellent henselian local regular $k$-algebra with infinite resiude field and it is a limit of regular local rings which are finite over $R$ and of degree prime to $\ell$. Thus if we prove the statement for $R_\ell$, it follows by the standard trace argument, that $\id-f_a^*$ on $G(\sX,\sZ)/G(\sX)$ is annihilated by an integer which is prime to $\ell$. Since  $\ell$ is arbitrary the statement follows. Therefore it suffices to assume in the following, that  $R$ is excellent, regular, local henselian with infinite residue field.

By the standard trace argument we may assume $R/\fm$ is infinite.
Set $S=\Spec R$, $X=\Spec R[t]$, and $Z=V(t)\subset X$. The pair $(X\to S,Z)$ is a nice V-pair 
by \cite[Lemma 5.2]{S-purity}. Let $x\in X$ be the point corresponding to $\fm+(t)$.
Let $A=R\{t\}$  and take $a\in A$.
By Lemma \ref{lem:gc-rel-curve} we find an $R[t]$-algebra $A_a\subset A$ satisfying the following conditions:
\begin{enumerate}[label = (\roman*)]
\item\label{prop:BK1}
$a\in A_a$ and $1+a t^{n-1}\in A_a^\times$;
\item\label{prop:BK2}
$X_a=\Spec(A_a)$ is a Nisnevich neighborhood of $(X,x)$;
\item\label{prop:BK3} the natural map $j_a:X_a\to X$ induces an isomorphism $Z_a:=j_a^{-1}(Z)\cong Z$;
\item\label{prop:BK4} $(X_a\to S, Z_a)$ admits a good compactification.
\end{enumerate}
By \cite[Lemma 4.2, Lemma  4.3]{S-purity}, $(X_a,Z_a)$ is a nice $V$-pair over $S$. 
Consider the $S$-morphism
\[ \phi_a :X_a \to X \quad \text{induced by } R[t]\to A_a, \quad  t\mapsto t+a t^n=t(1+a t^{n-1}).\]
By  \ref{prop:BK1} and \ref{prop:BK3}, $Z_a\cong \phi_a^{-1}(Z)$. 
%Furthermore
% $\phi_a$ is \'etale over a neighborhood of $Z$. (Indeed, since $d(t+at^n)= (1+t^{n-1}(t \partial a/\partial t + n a ))dt$ in
% $\Omega^1_{A_a/R}= A_a dt$, the map $\phi_a$ is unramified in a neighborhood of $Z_a$ and hence also
% \'etale, by \cite[Expos\'e I, Corollaire 9.11]{SGA1}.)
 Let
\[ \lambda :\sX \to X\quad \text{and} \quad \lambda_a : \sX\to X_a\]
be the maps induced by the natural inclusions 
\[ R[t]\to A\quad \text{and}\quad A_a\to A,\]
respectively. Note that $\lambda= j_a\circ \lambda_a$. 
We obtain the following commutative diagrams
\[\xymatrix{
\frac{G(X,nZ)}{G(X)} \ar[r]^{\lambda^*}\ar[d]^{j_a^*} 
&  \frac{G(\sX,n\sZ)}{G(\sX)} \ar[d]^{id} \\
\frac{G(X_a,nZ_a)}{G(X_a)} \ar[r]^{\lambda_a^*} &  \frac{G(\sX,n\sZ)}{G(\sX)} \\}
\qquad
\xymatrix{
\frac{G(X,nZ)}{G(X)} \ar[r]^{\lambda^*}\ar[d]^{\phi_a^*} 
&  \frac{G(\sX,n\sZ)}{G(\sX)} \ar[d]^{f_a^*} \\
\frac{G(X_a,nZ_a)}{G(X_a)} \ar[r]^{\lambda_a^*} &  \frac{G(\sX,n\sZ)}{G(\sX)}. \\
}
\]
The horizontal arrows are  isomorphisms by \cite[Corollary 2.21]{S-purity}
and Lemma \ref{lem:gc-rel-curve}.
Thus it suffices to show
\eq{prop:BK5}{
\phi_a^* = j_a^* : \frac{G(X,nZ)}{G(X)} \to \frac{G(X_a,nZ_a)}{G(X_a)}.
}
Set
\[ h_a= s - (t+a t^n) \in \Gamma(X_a\times_S X,\sO)=A_a[s],\]
where $X=\Spec R[s]$. Then $\Div(h_a)$ and $\Div(s-t)\subset X_a\times_S X$
are the graphs of $\phi_a$ and $j_a$, respectively.
Since furthermore $\Div(s-1)$ is the graph of the composition of the projection $X_a\to S$ with the 1-section 
$S\inj \A^1_S=X$, we see that $\Div(s-1)^*$ induces the zero map between the quotients in \eqref{prop:BK5}.
Thus  the equality in \eqref{prop:BK5} follows   by \cite[Theorem 2.10(2)]{S-purity} from the following claim.
\begin{claim}\label{prop:BK6}
$\frac{h_a}{s-1}$ and $\frac{s-t}{s-1}$ are admissible for 
$\big((X,nZ),(X_a,nZ_a)\big)$ (cf. \cite[Definition 2.3]{S-purity}), i.e., 
with $\theta_h=\Div(h)$, for $h\in\{\frac{h_a}{s-1}, \frac{s-t}{s-1}\}$, we have 
\begin{enumerate}[label=(\roman*)]
\item\label{prop:BK7}  $h$ is regular in a neighborhood of $X_a\times_S Z$; 
\item\label{prop:BK8}
$\theta_h\times_X nZ $ is the image of the diagonal map 
$nZ \to X_a\times_S X$ coming from the isomorphism $n Z_a\simeq nZ$ induced by $j_a$;
\item\label{prop:BK9}
$h$ extends to an invertible function on a neighborhood of $X_a\times_S \infty$ in 
$X_a\times_S \P^1_S$.
\end{enumerate}
\end{claim}
\ref{prop:BK7} is immediate in both cases. Note that $s-1$ is a unit on $X_a\times_S {nZ}$.
Hence \ref{prop:BK8} follows from the equality of ideals in $A_a[s, \frac{1}{s-1}]$ :
\[ (h,s^n) = (s-t,t^n,s^n). \]
This is immediate for $h=s-t/s-1$; for $h=h_a/s-1$ we set $u:=1+a t^{n-1}\in A_a^\times$ and then the above equality
follows from
\[t^n= (t-\tfrac{s}{u})(t^{n-1}+ t^{n-2}\tfrac{s}{u}+\ldots + (\tfrac{s}{u})^{n-1})+ (\tfrac{s}{u})^n\in (h_a, s^n).\]
Letting $\sigma=1/s$, \ref{prop:BK9} follows from 
\[ \tfrac{h_a}{s-1}= \tfrac{1 -\sigma(t+a t^n)}{1-\sigma} \quad \text{and} \quad 
\tfrac{s-t}{s-1}=\tfrac{1-\sigma t}{1-\sigma}  \quad\text{in } A_a[\sigma].\]
This proves Claim  \ref{prop:BK6} and completes the proof of Proposition \ref{prop:BK}.
\end{proof}

\begin{cor}\label{cor:BK}
%Let $(R, \fm)$ be a  regular henselian local $k$-algebra. Let $\sigma:\kappa\inj R$ be a  coefficient field, 
Let $K$ be a function field over $k$ and denote by (see Notation \ref{nota})
\[\pi:\sX:= \Spec K\{x,t\}\to T:=\Spec K\{t\}\]
the natural map induced by the inclusion  $K\{t\}\inj K\{x,t\}$,
and set $\sZ=V(t)\subset \sX$.
Let $a\in K\{t\}$ and $n\ge 2$. Consider the closed immersions
\[ i: C:=V(x-t)\inj \sX \quad \text{and}\quad i_a: C_a:=V(x-t-at^n)\inj \sX.\]
Then the compositions
\[\pi_C=\pi\circ i: C\to T \quad \text{and}\quad \pi_{C_a}=\pi\circ i_a: C_a\to T\]
are isomorphisms and for any $g\in G(\sX,n\sZ)$ we have
\[\pi_{C*}i^*g - \pi_{C_a*}i_a^*g\in G(T),\]
where $\pi_{C*}$ is induced by the action of the transpose of the graph of $\pi_C$, 
which is equal to $(\pi_C^*)^{-1}$, similarly with $\pi_{C_a*}$ .
\end{cor}
\begin{proof}
It is direct to check that $\pi_{C_a}$ and $\pi_C$ are isomorphisms.
%The map $\pi_{C_a}$ is an isomorphism since it is induced by
%\[K\{x,t\}/(x-t-at^n)\cong \left(K[x,t]/(x-t-at^n)\right)^h= K\{t\},\]
%and similarly with $\pi_C$. 
%Denote by $\bar{a}$ the residue class of $a$ in $R\{t\}/\fm R\{t\}= (R/\fm)\{t\}\cong \kappa\{t\}$ and by abuse of notation denote the map $\kappa\{t\}\inj R\{t\}$ induced by $\sigma$ again by $\sigma$. Then $a-\sigma(\bar{a})\in \fm R\{t\}$ and hence $C_{a}=C_{\sigma(\bar{a})}$. Therefore we may assume $a=\sigma(a_0)$, for some $a_0\in \kappa\{t\}\subset R\{t\}$. 
Let $f_a: (\sX, n\sZ)\to (\sX, n\sZ)$ be the $K\{x\}$-morphism induced by $t\mapsto t+at^n$.
Set $V(x-t, t^n)=:n\sZ_C\subset \sX$. Then 
$C\cap n\sZ=n\sZ_C= C_a\cap n\sZ=f_a^{-1}(C\cap n\sZ)$.
Thus $f_a$ restricts to a morphism 
$f_{C,a}: (C_a, n \sZ_C)\to (C, n\sZ_C)$. 
Denote by $0$ the closed point of $T$ and define the $K$-morphism
\[f_{T,a}: (T, n\cdot 0)\to (T, n\cdot 0) \quad \text{induced by}\quad  
K\{t\}\to K\{t\}, \quad t\mapsto t+at^n.\]
We have 
\[f_{T, a}\circ \pi_{C_a}= \pi_C \circ f_{C,a}: (C_a, n\sZ_C)\to (T, n\cdot 0)\]
and obtain the commutative diagram
\[\xymatrix{
\frac{G(\sX, n\sZ)}{G(\sX)}\ar[r]^{i^*}\ar[d]^{f_a^*}&
\frac{G(C, n\sZ_C)}{G(C)}\ar[r]^{\pi_{C*}}\ar[d]^{f_{C,a}^*}&
\frac{G(T, n\cdot 0)}{G(T)}\ar[d]^{f_{T, a}^*}\\
\frac{G(\sX, n\sZ)}{G(\sX)}\ar[r]^{i_a^*}&
\frac{G(C_a, n\sZ_C)}{G(C_a)}\ar[r]^{\pi_{C_a*}}&
\frac{G(T, n\cdot 0)}{G(T)}.
}\]
By Proposition \ref{prop:BK} we find that $f_a^*$  (resp. $f_{T, a}^*$ )
induces the identity on the left hand side (resp. on the right hand side) in the above diagram.
This yields the statement.
\end{proof}

\begin{rmk}\label{rmk:BK}
Let $L$ be a henselian dvf of geometric type over $k$ with $\sO_L$ the ring of integers and $\fm$ the maximal ideal.
Recall from \cite[Def 4.14]{RS} that we can associate a conductor to $G$ which at $L$ is the
function $c_L^G: G(L)\to \N_0$ given by (see \eqref{para:RSC2} for notation)
\[c^G_L(g)= \min\{n\ge 0\mid g\in G(\sO_L, \fm^{-n})\}.\]
In the situation of Corollary \ref{cor:BK}, let $g\in G(\sX, n\sZ)$.
The intersections $C\cap\sZ=C_a\cap \sZ$ are equal to the closed point of $\sX$ and 
the intersection multiplicity of $C$ and $C_a$ at this closed point is $\ge n$ and the corollary implies that 
$c^G_L(g_{|C})= c^G_{L_a}(g_{|C_a})$, where $L$ (resp. $L_a$) is the function field of $C$ (resp. $C_a$).

In case $G=\widetilde{H^1}$ (see \eqref{para:RSC1.5} for notation), where
$H^1$ is the reciprocity sheaf $X\mapsto \Hom_{\rm cont}(\pi_1(X)^{\rm ab}, \Q/\Z))$ (see \cite[8.1]{RS}), Corollary \ref{cor:BK} gives back a special case of an Artin-version of the Brylinski-Kato formula in 
\cite[Theoremm 0.2]{KatoSwan} (see also \cite[Proposition 4]{Brylinski}).
It is a special case since in \cite{KatoSwan} the curves are only assumed to intersect the divisor
$\sZ$ transversally and  have intersection multiplicity $\ge n$, but do not need the specific form
as in the corollary above; it is the Artin version, since in \cite{KatoSwan} the statement is for the Swan conductor
whereas $c^{H^1}$ is equal to the Artin  conductor, by \cite[Theorem 8.8]{RS}.
Note that by \cite{RS} we can apply Corollary \ref{cor:BK} also for the (non-log) irregularity of rank one connections 
(in char $0$), the Artin conductor of lisse $\ol{\Q}_\ell$-sheaves, the conductor of fppf-torsors under 
finite $k$-group schemes etc. We will prove the Brylinsky-Kato formula in full generality in \cite{RS-AS} (without the restriction on the special choice of $C$ and $C_a$).
\end{rmk}

\begin{para}\label{para:adj-pairing}
%(We use the notations from  \eqref{para:twist1}, \eqref{para:twist2}.)
By \cite[4.3]{BRS} we have a functor
\eq{para:adj-pairing0}{h^{\bcube,\sp}_{0,\Nis}: \uMPST^{\tau}:=\tau_!(\MPST)\to \CItspNis,}
which is the sheafified and semi-purified maximal cube invariant quotient  and is 
left-adjoint to the inclusion functor $\CItspNis\to\uMPST^{\tau}$.
We set 
\eq{para:adj-pairing2.5}{\gamma^1G:= \uHom_{\uMPST}(\uomega^*\G_m, G), \quad 
G(1):=h_{0,\Nis}^{\bcube,\sp}(G\otimes_{\uMPST}\uomega^*\G_m),}
where $(\uomega^*\G_m)(X,D)=\G_m(X\setminus D)$.
By the above and \cite[Lemma 1.9(1)]{BRS} we have $\gamma^1G$, $G(1)\in \CItspNis$.
\footnote{The notation $\gamma^1 G$ is compatible with \eqref{para:twist2},  in the sense that $\uomega_!\gamma^1 G=\gamma^1\uomega_!G$,  
by \cite[Proposition 2.10]{MS}.}
Consider the natural pairing on $\uMCor$
\eq{para:adj-pairing1}{\uHom_{\uMPST}(\uomega^*\G_m, G)\times \uomega^*\G_m\to 
\uHom_{\uMPST}(\uomega^*\G_m, G)\otimes_{\uMPST} \uomega^*\G_m\to G,}
where the first map is the natural one and 
the second map is the counit of adjunction. Since $G\in \CItspNis$,  this composition induces a pairing
\eq{para:adj-pairing0.5}{\gamma^1 G \times \uomega^*\G_m\to (\gamma^1 G)(1)\to G.}
%Applying $\uomega_!$  we obtain a pairing 
%\eq{para:adj-pairing2}{\uomega_!\gamma^1 G\times \G_m\to \uomega_!(\gamma^1G)(1))\to \uomega_!G,}
%where we use the notations from  \eqref{para:twist1}, \eqref{para:twist2} and 
%\eq{para:adj-pairing2.5}{\gamma^1G:= \uHom_{\uMPST}(\uomega^*\G_m, G), \quad 
%G(1):=h_{0,\Nis}^{\bcube,\sp}(G\otimes_{\uMPST}\uomega^*\G_m)}
%and the facts  $\uomega_!G\in \NST$ (cf. \ref{para:RSC}) and 
%$\uomega_!\uomega^*\G_m=\G_m$ (by definition).
By \cite[(2.2) and Corollary 2.2]{MS} we have for $\sX= (X,D)\in \uMCor$
\eq{para:adj-pairing3}{\gamma^1 G(\sX)= \frac{G(\sX\otimes (\P^1, 0+\infty))}{p_2^* G(\sX)},}
where $p_2: \sX\otimes (\P^1, 0+\infty)\to \sX$ is induced by the projection $X\times\P^1 \to X$ (Note that $\gamma^1 G(\sX)$ is a direct summand of $G(\sX\otimes (\P^1, 0+\infty))$).
For a henselian regular local $k$-algebra $R$ of geometric type, \eqref{para:adj-pairing0.5} induces the bilinear pairing (cf. \eqref{notation2})
\eq{para:adj-pairing4}{\gamma^1G(R)\times R^\times \to (\gamma^1G)(1)(R)\to G(R).}
By \cite[Lemma 5.6]{BRS} the precomposition of \eqref{para:adj-pairing4} with
the natural map $G(\P^1_R, 0_R+\infty_R)\times R^\times\to \gamma^1G(R)\times R^\times$ is given by
\eq{para:adj-pairing5}{G(\P^1_R, 0_R+\infty_R)\times R^\times\to G(R), \quad (a, f)\mapsto \delta_f^*a-\delta_1^*a,}
where $\delta_f: \Spec R\to \Spec R\times \G_m$ denotes the graph of the $k$-morphism $f\in \Hom(\Spec R, \G_m)$.
\end{para}

\begin{lem}\label{lem:cont}
Let $(\sO_L,\fm)$ be a henselian dvr of geometric type over $k$.
Then the pairing 
\[(\gamma^1 G)(L)\times L^\times \xr{\eqref{para:adj-pairing4}} G(L)\]
 restricts to a pairing
\[(\gamma^1G)(\sO_L, \fm^{-n})\times U_L^{(n)}\to G(\sO_L), \quad n\ge 1,\]
where $U_L^{(n)}=1+\fm^n$.
\end{lem}
\begin{proof}
The choice of a coefficient field $K\inj \sO_L$ and a local parameter $t\in\fm$ yields an isomorphism
$\sO_L\cong K\{t\}$.
Set $T=\Spec K\{t\}$ and denote by $0$ the closed point.
Since any element in $U_L^{(n)}$ can be written in the form
$(1+t+a t^n)/(1+t)$, with $a\in K\{t\}$, it suffices to show 
\eq{lem:cont1}{\delta_{1+t+at^n}^*=\delta_{1+t}^*: G((T,n\cdot 0)\otimes (\P^1, 0+\infty))\to  
\frac{G(T\setminus 0)}{G(T)}.}
Write $\P^1\setminus\{\infty\}=\Spec k[y]$ and set $x=y-1$.
Then the henselization of the local ring of $\P^1_T$ at $(t,x)$ is identified with
$K\{x,t\}$. Let $\sX=\Spec K\{x,t\}$ and $\sZ=\{t=0\}\subset \sX$.
Then $\delta_{1+t+at^n}$ is the composition 
\[T\xr{\simeq, \,\pi_{C_a}^{-1}} C_a\xr{i} \sX\to \P^1_T,\]
where we use  the notation from Corollary \ref{cor:BK} and where the last map is the natural one;
for $\delta_{1+t}$ the same holds with $C_a$ replaced by $C$ and $i_a$ replaced by $i$.
Hence \eqref{lem:cont1} follows from this corollary.
\end{proof}

The following is the main result of this section,
which will play a key role in the proof of the Zariski-Nagata purity Theorem \ref{thm:ZNP}.

\begin{thm}\label{thm:cont}
Let $F\in \RSC_\Nis$ and let $(\sO_L,\fm)$ be a henselian dvr of geometric type over $k$.
Then  (with the  notation from \ref{para:RSC}, \ref{para:twist}) the pairing 
\[F(L)\times L^\times\to (F\otimes_{\NST} \G_m)(L)\to F\la 1 \ra(L)\]
 induces a pairing
\[\tF(\sO_L,\fm_L^{-n}) \times U^{(n)}_L\to F\la 1\ra(\sO_L).\]
\end{thm}
\begin{proof}
Consider the diagram
\[\xymatrix{
\tF\times \uomega^*\G_m\ar[d]_{\ol{{\rm unit}}\otimes \id_{\G_m}}\ar[r] & 
\tF(1)\ar[d]_{\ol{{\rm unit}}(1)}\ar@{=}[drr] & &\\
\gamma^1(\tF(1)) \times \uomega^*\G_m\ar[r] & (\gamma^1(\tF(1))(1)\ar[rr]_-{\ol{\text{counit}}_{\tF(1)}} & &
\tF(1).
}\]
where we use the notation from \eqref{para:adj-pairing2.5},  
the horizontal pairings are induced  by the natural map for $G\in \uMPST$:
\eq{thm:cont1}{G\otimes_{\uMPST} \uomega^*\G_m\to 
h_{0,\Nis}^{\bcube,\sp}(G\otimes_{\uMPST} \uomega^*\G_m)=G(1),}
the map $\ol{\text{counit}}_{\tF(1)}$ is  induced by the counit of the adjunction 
\[(-)\otimes_{\uMPST} \uomega^*\G_m\dashv \uHom_{\uMPST}(\uomega^*\G_m,-)\]
(see \eqref{para:adj-pairing0.5}), and the vertical maps are induced by 
\[G\xr{\rm unit} \uHom_{\uMPST}(\uomega^*\G_m, G\otimes_{\uMPST} \uomega^*\G_m)\xr{\eqref{thm:cont1}}
\uHom_{\uMPST}(\uomega^*\G_m, G(1)),\]
where ``unit'' denotes the unit of the above adjunction.
We claim the diagram commutes. Indeed for the square this holds by functoriality and for the triangle,
this follows from the fact that the composition 
\[\tF\otimes \uomega^*\G_m\xr{\text{unit}\otimes \id} 
\uHom(\uomega^*\G_m, \tF\otimes\uomega^*\G_m)\otimes\uomega^*\G_m
\xr{\text{counit}_{\tF\otimes\uomega^*\G_m}} \tF\otimes\uomega^*\G_m,\]
where all $\otimes$ and $\uHom$ are in $\uMPST$,  is the identity for general reasons, e.g. \cite[IV, Theorem 1]{MacLane}.
% the map ${\rm adj}(1)$ on the right is an isomorphism by
%\cite[Corollary 3.6, Lemma 1.14, Corollary 2.2]{MS}.
We have:
\begin{enumerate}[label=(\roman*)]
\item $F\la 1\ra =\uomega_! (F(1))$, by definition \cite[5.21]{RSY};
\item $\uomega_! \gamma^1 G=\uomega_!\uHom_{\uMPST}(\uomega^*\G_m,G)=\uHom_{\PST}(\G_m,\uomega_!G)
= \gamma^1\uomega_! G$ for $G\in \uMPST$,
where the middle equality holds by \cite[Proposition 2.10]{MS}.
\end{enumerate}
Thus, applying $\uomega_!$ to the diagram above and evaluating on $L$ gives the diagram
\[\xymatrix{
F(L)\times L^\times \ar[r]\ar[d] & F\la 1\ra(L)\ar[d]\ar@{=}[dr]& \\
\gamma^1(\tF(1))(L) \times L^\times\ar[r] & \left(\gamma^1(\tF(1)))(1)\right)(L)\ar[r] & F\la 1\ra(L),
}\]
where we use the notation \eqref{notation2} in the bottom line.
Hence the statement follows from Lemma \ref{lem:cont} with $G=\tF(1)\in \CItspNis$.
\end{proof}

\begin{rmk}\label{rmk:adj-pairing}
We remark that the pairing \eqref{para:adj-pairing4} can be trivial.
Indeed, e.g., by the projective bundle formula \cite[Theorem 6.3]{BRS} 
we have  $\gamma^1\tilde{\G}_a(R)=H^1(\P^1_R,\sO_{\P^1_R})=0$,  with $\G_a$ the additive group.
However, the pairing \eqref{para:adj-pairing4} for $G=\tilde{\G_a}(1)$ is non-trivial.
Indeed, assume char$(k)=0$, then $\tilde{\G}_a(1)=\widetilde{\Omega^1_{-/\Z}}$ by \cite[Theorem 11.1]{BRS}.
And using \cite[Theorem 6.4]{RS} and \eqref{para:adj-pairing3} we find (for $n\ge 1$)
\[\gamma^1(\tilde{\G}_a(1))(\sO_L, \fm^{-n})= \frac{1}{t^{n-1}}\cdot\sO_L \cdot \dlog (y),\]
where $t$ is the local parameter of $\sO_L$ and $y$ is the standard coordinate of $\P^1$,
and the pairing \eqref{para:adj-pairing5} is is given by
\[(\tfrac{1}{t^{n-1}} a\cdot \dlog(y), f)\mapsto \tfrac{1}{t^{n-1}} a\dlog(f), \]
which is clearly in $\Omega^1_{\sO_L/\Z}$, if $f\in U^{(n)}_L$.
\end{rmk}

\section{Pushforwards with compact supports}\label{sec:pfcs}
In this section $k$ is a perfect field. We fix $F\in \RSC_\Nis$. For $X\in \Sm$ we denote by $F_X$ 
the restriction of $F$ to $X_{\Nis}$.

\begin{para}\label{para:proj-pf}
Let $f: X\to Y$ be a projective morphism in $\Sm$ of relative dimension $r\in\Z$.
In \cite[9.5]{BRS} there is constructed a pushforward 
\eq{para:proj-pf1}{ f_* : Rf_*(\gamma^b F\la a+r\ra_{X})[r] \to \gamma^b F\la a\ra_Y,}
for all $a,b\ge 0$ with $a+r\ge 0$, in $D(Y_\Nis)$ the derived category of abelian Nisnevich sheaves on $Y$.
Recall that the pushforward  is essentially given by 
a Gysin map for a closed immersion of $X$ into a projective bundle over $Y$
followed by a projective trace map stemming from the projective bundle formula.
By \cite[Thm 9.7]{BRS} this pushforward has all the expected properties.
\end{para}

\begin{para}\label{para:ext-zero}
Let $\iota: Z\inj X$  be a locally closed immersion between  separated schemes.
For $G$ a sheaf of abelian groups on $Z_\Nis$, we define $\iota_!G$ as the subsheaf of $\iota_* G$ on $X_\Nis$
given on $V\to X$ \'etale by
\[\iota_! G(V)=\{s\in G(Z\times_X V)\mid {\rm supp}(s) \text{ is proper over }V\}.\]
It follows from \cite[Exp. XVII, Prop 6.1.1]{SGA4III} that $\iota_!G$ is a Nisnevich sheaf on $X$.
If we choose a factorization  
\eq{para:ext-zero1}{\iota: Z\xr{i} U\xr{j} X,}
with $i$ a closed immersion and $j$ an open embedding, it follows directly from the definition, that we have 
\[\iota_! G =j_! i_* G.\]
Hence, for $V\to X$ \'etale and $v\in V$ a point  
\[\iota_! G(\sO_{V,v}^h)=
\begin{cases} G(\sO_{V_Z,v}^h), & \text{if } v\in V_Z=V\times_X Z,\\ 0, & \text{else.}  \end{cases}\]
From this formula we see that $\iota_!$ is exact and furthermore, if $\kappa: W\inj Z$ is another
locally closed immersion with $W$ separated and $H$ is a sheaf on $W_\Nis$, then there is a canonical identification 
\[(\iota\kappa)_! H = \iota_!\kappa_! H \quad \text{inside } (\iota\kappa)_*H=\iota_*\kappa_*H.\]
Since $\iota_!$ is an exact functor on the category of abelian Nisnevich sheaves, 
it extends to an exact triangulated functor between the derived categories
\[\iota_! :D(Z_\Nis)\to D(X_\Nis).\]
If $j: U\inj Z$ is open, then $j_!$ is equal to the  extension-by-zero functor from \cite[Exp. IV, Prop 11.3.1]{SGA4I} which is left adjoint to $j^{-1}$. 
The adjunction extends to an adjunction 
$j_!: D(U_\Nis)\leftrightarrows D(X_\Nis): j^{-1}$.

\begin{lem}
For a commutative diagram 
\eq{para:ext-zero1a}{
\xymatrix{
  Z'\ar[d]^g\ar[r]^\kappa & X'\ar[d]^f\\
Z\ar[r]^\iota & X
}}
with $f$ and $g$ proper and $\iota$, $\kappa$ locally closed immersions, 
we have a natural equivalence  
\eq{para:ext-zero2}{\iota_!g_*\isom f_*\kappa_!: {\rm Sh}(Z'_\Nis)\to {\rm Sh}(X_\Nis).}
\end{lem}
\begin{proof}
Let $V\to X$ be \'etale 
and $s\in \iota_*g_*G(V)=G(V\times_X Z')= G(V'\times_{X'} Z')=f_*\kappa_*G(V)$ a section,
where $V'=f^{-1}(V)=V\times_X X'$. Assume $s\in \iota_!g_*G(V)$, i.e., viewing $s$ as a section in $g_*G(V\times_X Z)$, its support $S\subset  V\times_X Z$ is proper over $V$. Viewing $s$ as a section in $G(V\times_X Z')$, its support is equal 
to $S'=S\times_Z Z'$. As $g$ is proper, the composition $S'\subset V\times_X Z'\to V\times_X Z\to V$ 
is also proper.  Since this last map factors as 
$V\times_X Z'=V'\times_{X'} Z'\to V'\xr{f_{|V'}} V$, we conclude that $S'$ 
is proper over $V'$. This defines the map \eqref{para:ext-zero2}. 
On the other hand, if $s\in f_*\kappa_!G(V)$, i.e.,  $S'$ is proper over $V'$, 
then $S'$  is also proper over $V$, as $f$ is proper, whence $S$ is proper over $V$ as well, i.e., $s\in \iota_!g_*G(V)$. 
This shows that \eqref{para:ext-zero2} is an isomorphism.
\end{proof}

In \eqref{para:ext-zero1a}, if $C$ is a complex of Nisnevich sheaves on $Z'_{\rm Nis}$, choose a $K$-injective resolution $C\to I$ and a $K$-injective 
resolution $\kappa_!I\to J$. We obtain a natural morphism 
$\iota_!g_*I\to f_*\kappa_!I\to f_*J$ which yields a natural transformation
\eq{para:ext-zero3}{\alpha_\sigma:\iota_!Rg_*\to Rf_* \kappa_!: D(Z'_\Nis)\to D(X_\Nis),}
where $\sigma$ denotes the square \eqref{para:ext-zero1a}.
It follows directly from the construction that this transformation is functorial in the following sense:
assume given another square $\sigma'$ as \eqref{para:ext-zero1a}:
\[\xymatrix{
  T'\ar[d]^h\ar[r]^\mu & Z'\ar[d]^g\\
T\ar[r]^\lambda & Z,
}\]
then we have the equality
\eq{para:ext-zero4}{\alpha_{\sigma\circ \sigma'}= \alpha_{\sigma}\circ \alpha_{\sigma'}:
(\iota\lambda)_!Rh_*\to \iota_!Rg_*\mu_! \to Rf_* (\kappa\mu)_!,}
where $\sigma\circ \sigma'$ denotes the square obtained by concatenating $\sigma$ and $\sigma'$ in the obvious way.
Furthermore, if $\iota$ and $\kappa$ are closed immersions, \eqref{para:ext-zero3} is equal to the isomorphism
$\iota_*Rg_*\cong Rf_*\kappa_*$.
\end{para}

\begin{defn}\label{defn:gysin}
Let $\iota: Z\inj X$  be a locally closed immersion in $\Sm$ of pure codimension $c$.
For $a\ge c$ and $b\ge 0$  we define the {\em Gysin map} for $Z\inj X$ by
\[g_{Z/X}: \iota_! \gamma^b F\la a-c\ra_Z [-c]\to \gamma^b F\la a\ra_X\quad  \text{in } D(X_\Nis) \]
by choosing a factorization \eqref{para:ext-zero1} and defining  $g_{Z/X}$ as the composition 
\[j_!(i_*\gamma^b F\la a-c\ra_Z [-c])\xr{j_! (i_*)} j_!(\gamma^b F\la a\ra_U)\xr{\text{nat}} \gamma^b F\la a\ra_X,\]
where $i_*$ denotes the pushforward from \eqref{para:proj-pf1} for $f=i$ and 
the map ``nat'' denotes the counit of the adjunction $j_!\dashv j^{-1}$.
This definition is independent of the factorization \eqref{para:ext-zero1} by Lemma \ref{lem:gysin}\ref{lem:gysin1} below.
\end{defn}

\begin{lem}\label{lem:gysin}
Let $ \iota: Z\inj X$ be as in Definition \ref{defn:gysin} above.
\begin{enumerate}[label=(\arabic*)]
\item\label{lem:gysin1} $g_{Z/X}$ is independent of the choice of the factorization \eqref{para:ext-zero1}.
\item\label{lem:gysin2}
Let $\lambda: W\to Z$ be a closed immersion in $\Sm$ of pure codimension $d$ and assume $a\ge c+d$, $b\ge 0$.
Then the following diagram commutes:
\[\xymatrix{
\iota_!\lambda_* \gamma^b F\la a-d-c\ra_W[-d-c] \ar[r]^-{\iota_! g_{W/Z}} \ar@{=}[d]& 
    \iota_! \gamma^b F\la a-c\ra_Z[-d]\ar[d]^{g_{Z/X}}\\
(\iota\lambda)_! \gamma^b F\la a-d-c\ra_W [-d-c]\ar[r]^-{g_{W/X}} & \gamma^b F\la a\ra_X }.\]
\item\label{lem:gysin3}
Let $\tau: Z'\to Z$ be an open immersion and $a\ge c$, $b\ge 0$. Then the following diagram commutes:
\[\xymatrix{
\iota_!\tau_! \gamma^b F\la a-c\ra_{Z'}[-c] \ar[r]^-{g_{Z'/X}} \ar[d] & \gamma^b F\la a\ra_X \\
\iota_!  \gamma^b F\la a-c\ra_Z[-c] \ar[ru]_-{g_{Z/X}} \\},\] 
where the vertical map is induced by the counit
$\tau_! \tau^{-1}\to \id$.
\item\label{lem:gysin4}
Assume we have a commutative triangle
\[\xymatrix{
    & P\ar[d]^{\pi}\\
Z\ar[ur]^{\kappa}\ar[r]^{\iota} & X
}\]
with $\pi$ smooth projective of relative dimension $n$ and $\iota$ and $\kappa$ locally closed immersions 
of pure codimension $c$ and $n+c$, respectively.
Then the following diagram commutes for $a\ge c$, $b\ge 0$
\[\xymatrix{
\iota_!\gamma^b F\la a-c\ra_Z[-c]\ar[r]^{g_{Z/X}}\ar[d]_{\eqref{para:ext-zero3}} & \gamma^b F\la a\ra_X\\
R\pi_*\kappa_! \gamma^b F \la a-c\ra_Z[-c]\ar[r]^{g_{Z/P}} & R\pi_*\gamma^b F\la a+n\ra_P[n]\ar[u]^{\pi_*}
}\]
\end{enumerate}
\end{lem}
\begin{proof}
In the following we  write for short $G(n)_P:=\gamma^b F\la a+n\ra_P[n]$.

\ref{lem:gysin1}. Let $Z\rmapo {i'} U' \rmapo {j'} X$ be another factorization.
If $j$ factors through an open immersion $\tau: U\to U'$, the assertion is reduced to the commutativity of 
\[\xymatrix{
\tau_! \tau^{-1}i'_* G(-c)_Z\ar[r]^-{\tau_!(i_*)}\ar[d]&
\tau_! \tau^{-1} G(0)_{U'}\ar[d]\\
i'_* G(-c)_Z\ar[r]^-{i'_*} & G(0)_{U'},}\]
where the vertical maps are induced by $\tau_!\tau^{-1}\to \id$ and we use $\tau^{-1}i'_*= i_*$ by \cite[Thm 9.7(3)]{BRS}.
Since the top arrow is equal to $\tau_!\tau^{-1}(i'_*)$, the diagram clearly commutes.
The general case is reduced to the above case by considering the factorization
$Z\to U_1\cap U_2 \to X$.
To show \ref{lem:gysin2} we can choose the factorization $W\xr{i\lambda} U \xr{j} X$ to compute 
$g_{W/X}$. Hence the commutativity of the diagram follows directly from the equality of pushforwards  
$(i\lambda)_*=i_*\circ \lambda_*$, which holds by \cite[Thm 9.7(2)]{BRS}.
To show \ref{lem:gysin3} consider the closed immersion $i': Z' \to U'=U\setminus(Z\setminus Z')$
and  the open immersion $\tau_U:U'\to U$.
We can  thus compute $g_{Z'/X}$ using the factorization
$Z'\xr{i'} U'\xr{j\tau_U} X$.  The following diagram is clearly commutative 
\[\xymatrix{
j_! \tau_{U!}\tau_U^{-1} i_* G(-c)_Z\ar[r]^-{i_*}\ar[d]&
j_!\tau_{U!} \tau_{U}^{-1} G(0)_U\ar[r]\ar[d]&
G(0)_X\\
j_! i_* G(-c)_Z[-c]\ar[r]^-{i_*}&
j_! G(0)_U,\ar[ur]&
}\]
where the vertical maps are induced by the counit $\tau_!\tau^{-1}\to \id$.
Since the pushforward is compatible with restriction to opens (see \cite[Thm 9.7(3)]{BRS}),
the top row in the diagram is equal to the composition 
\[(j\tau_U i')_!G(-c)_{Z'}\xr{i'_*}(j\tau_U)_! G(0)_{U'}\to G(0)_X.\]
This proves \ref{lem:gysin3}. 
Finally we prove \ref{lem:gysin4}. A factorization \eqref{para:ext-zero1} induces a diagram with cartesian square
\[\xymatrix{
         &   P_U\ar[r]^{j'}\ar[d]^{\pi_U}    & P\ar[d]^{\pi}\\
Z\ar[ur]^{i'}\ar[r]^i & U\ar[r]^j &X,
}\]
with $i$, $i'$ closed - and $j$, $j'$ open immersions. 
By \eqref{para:ext-zero4} the map \eqref{para:ext-zero3} for the diagram
\[\xymatrix{
Z \ar[r]^-{\kappa}\ar[d]^{=} & D\ar[d]^{\pi} \\
Z\ar[r]^-{\iota} & X\\}\]
is given by
\eq{lem:gysin5}{\iota_!=j_!i_*=j_!R\pi_{U*}i'_*\to R\pi_* j'_!i'_*=R\pi_*\kappa_!,}
where the middle map is by adjunction induced from the natural isomorphism $R\pi_{U*}\xr{\simeq} j^{-1}R\pi_* j'_!$.
Consider the following diagram 
\[
\xymatrix{
j_!i_*G(-c)_Z\ar[r]^{i_*}\ar@{=}[d] &
j_!G(0)_U\ar[r] &
G(0)_X\\
j_! R\pi_{U*}i'_* G(-c)_Z\ar[d]\ar[r]^{i'_*}&
j_! R\pi_{U*} G(n)_{P_U}\ar[u]^{\pi_{U*}}\ar[r]\ar[d] &
R\pi_*G(n)_P\ar[u]^{\pi_*}\ar@{=}[d]\\
R\pi_* j'_!i'_*G(-c)_Z\ar[r]^{i'_*}&
R\pi_*j'_!G(n)_{P_U}\ar[r]&
R\pi_*G(n)_P.
}
\]
The square in the top left corner commutes by functoriality of the pushforward (see \cite[Thm 9.7(2)]{BRS});
the square on the top right corner commutes since $\pi_{U*}=j^{-1}(\pi_*)$ 
%(compatibility of the pushforward with restriction along smooth maps,  
(see \cite[Thm 9.7(3)]{BRS});
the other squares clearly commute. Thus the big outer square commutes which yields the statement.
\end{proof}

\begin{defn}\label{defn:pfcs}
Let $f:X\to Y$ be a quasi-projective morphism in $\Sm$ of pure relative dimension $r\in \Z$. We fix a compactification
\eq{defn:pfcs1}{f: X\xr{j} \ol{X}\xr{\ol{f}} Y,}
where $j$ is a dense open immersion and $\ol{f}$ is projective.
For $a\ge 0$ with $a+r\ge 0$ and $b\ge 0$ we define a pushforward
\[(\ol{f},j)_*: R\ol{f}_* j_! \gamma^b F\la a+r\ra_X[r]\to \gamma^b F\la a\ra_Y \]
as follows: choose a factorization 
\eq{defn:pfcs2}{ \ol{f}: \ol{X}\xr{i} P\xr{p} Y,}
with $i$ a closed immersion and $p$ projective smooth of relative dimension $n$.
%(This is always possible, see e.g., \cite[Exp. II, Cor 2.2.7.1]{SGA6}.) 
Then $(\ol{f},j)_*$ is defined as the composition
\ml{defn:pfcs3}{(\ol{f},j)_*: R\ol{f}_*j_!\gamma^b F\la a+r\ra_X[r]= Rp_* \iota_!\gamma^b F\la a+r\ra_X[r]\\
\xr{g_{X/P}} Rp_*\gamma^b F\la a+n\ra_P[n] \xr{p_*} \gamma^b F\la a\ra_Y,}
where $\iota=i\circ j: X\inj P$ denotes the immersion, $g_{X/P}$ denotes the associated Gysin map from \ref{defn:gysin},
and the map $p_*$ is the pushforward recalled in \ref{para:proj-pf}.
By Lemma \ref{lem:pfcs} below $(\ol{f},j)_*$ does not depend on the choice \eqref{defn:pfcs2}, however
it might depend on the choice \eqref{defn:pfcs1}. 
\end{defn}

\begin{lem}\label{lem:pfcs}
Assumptions as in Definition \ref{defn:pfcs}.
\begin{enumerate}[label=(\arabic*)]
\item\label{lem:pfcs1} The map $(\ol{f}, j)_*$ is independent of the choice of the factorization \eqref{defn:pfcs2}.
\item\label{lem:pfcs2} Assume $j$ factors  as a composition of open immersions
\[X\xr{\tau} X'\xr{j'}\ol{X}.\]
Then the following diagram commutes:
\[\xymatrix{
R\ol{f}_* j_! \gamma^b F\la a+r\ra_X[r]\ar[r]^-{(\ol{f}, j)_*}\ar[d] & \gamma^b F\la a\ra_Y\\
R\ol{f}_* j'_! \gamma^b F\la a+r\ra_{X'}[r]\ar[ru]_-{(\ol{f}, j')_*}, &
}\]
where the vertical map is induced by the natural transformation
\[j_!\tau^{-1} = j'_!\tau_! \tau^{-1}\to j'_! .\]
\item\label{lem:pfcs3} Assume given a commutative diagram
\[\xymatrix{
X' \ar[r]^{j'}\ar[d]^{g} & \ol{X'}  \ar[r]^{\ol{f'}}\ar[d]^{\ol{g}} & Y\\
X\ar[r]^{j} &\ol{X} \ar[ru]_{\ol{f}} 
}\] 
with $X,X',Y\in \Sm$, $j$, $j'$  dense open immersions and $g$, $\ol{g}$, $\ol{f}$, $\ol{f'}$ projective.
Set $s=\dim X'-\dim X$. Assume $a\ge \max\{0, -r,-r-s\}$. Then the following diagram commutes
\[\xymatrix{
R\ol{f}_*j_! Rg_* \gamma^b F\la a+r+s\ra_{X'}[r+s]\ar[r]^-{g_*}\ar[d]^{\eqref{para:ext-zero3}} &
R\ol{f}_*j_!  \gamma^b F\la a+r\ra_X[r]\ar[d]^{(\ol{f},j)_*}\\
R\ol{f'}_*j'_!  \gamma^b F\la a+r+s\ra_{X'}[r+s]\ar[r]^-{(\ol{f'}, j')_*} &
\gamma^b F\la a\ra_Y,
}\]
where $g_*$ is the pushforward from \eqref{para:proj-pf1}.
\end{enumerate}
\end{lem}
\begin{proof}
In the proof we set 
\[G(m)_P:= \gamma^b F\la a+m\ra_P[m].\]
\ref{lem:pfcs1}. Let $\ol{X}\xr{i_P} P\xr{p} Y$, and $\ol{X}\xr{i_Q} Q\xr{q} Y$, be two factorizations of $\ol{f}$
with $i_P$ and $i_Q$ closed immersions and $p$ and $q$ smooth projective morphisms of relative dimension $m$ and $n$,
respectively. We obtain the following commutative diagram
\[\xymatrix{
               &                                                        &  P\ar[dr]^{p}       &   \\
 X\ar[r]^j &\ol{X}\ar[ur]^{i_P}\ar[dr]_{i_Q}\ar[r] & PQ\ar[u]^{\pi_P}\ar[d]_{\pi_Q}\ar[r]^{\varrho}  & Y,\\
               &                                                        & Q\ar[ur]_{q}
        }\]
with $PQ=P\times_Y Q$ and $\pi_P$, $\pi_Q$ the projections.
This yields a diagram
\[\xymatrix{
    &   
    Rp_{*}G(m)_{P}\ar[dr]^-{p_*}     & 
    \\
R\ol{f}_*j_!G(r)_X\ar[ur]^-{g_{X/P}}\ar[r]^-{g_{X/PQ}}\ar[dr]_-{g_{X/Q}} &
R\varrho_*G(m+n)_{PQ}\ar[u]^{\pi_{P*}}\ar[d]_{\pi_{Q*}} &
G(0)_Y.\\ 
    &
Rq_*G(n)_{Q}\ar[ur]_-{q_*}&
  \\
}\]
The big triangle on the right hand side commutes by the functoriality of the pushforward, see \cite[Thm 9.7(2)]{BRS} and
the two small triangles on the left hand side commute by Lemma \ref{lem:gysin}\ref{lem:gysin4}.
This implies the statement.

\ref{lem:pfcs2} follows from \ref{lem:pfcs1} and Lemma \ref{lem:gysin}\ref{lem:gysin3}.
Finally \ref{lem:pfcs3}. We choose a factorization $\ol{X'}\to P'\to Y$ of $\ol{f'}$ as in \eqref{defn:pfcs2}.
After replacing $P'$ by $P\times_Y P'$ we obtain a commutative diagram
\[\xymatrix{
X'\ar[d]_g\ar@{^(->}[r]^{i'}\ar[dr]^h & U'\ar[d]^{\pi_U}\ar@{^(->}[r]^{k'} & 
P'\ar[d]_{\pi}\ar[r]^-{p'}& Y\\
X\ar@{^(->}[r]^i & U\ar@{^(->}[r]^k & P\ar[ur]_p & ,
}\]
where $p$, $p'$, $\pi$, $\pi_U$ are smooth projective, $k$, $k'$ are open immersions, $i$, $i'$ are closed immersions,
the square in the middle is cartesian, $h=\pi_U i'=ig$, and 
$\ol{f'}j'=p'k'i'$, $\ol{f}j=pki$.
Let $r$, $s$, $n$, and $n'$ be the relative dimensions of $f$, $g$, $p$, and $p'$, respectively, and set $t:=r+s$.
Consider the following diagram
\[\xymatrix{
Rp_* k_! i_*G(r)_X\ar[r]^{i_*}&
Rp_* k_! G(n)_U\ar[r]^{\rm nat}&
Rp_*G(n)_P\ar[r]^{p_*}&
G(0)_Y\\
Rp_*k_!Rh_* G(t)_{X'}\ar[u]^{g_*}\ar[r]^{i'_*}\ar[d]_{\eqref{para:ext-zero3}}&
Rp_*k_! R\pi_{U*}G(n')_{U'}\ar[r]^{\rm nat}\ar[d]_{\eqref{para:ext-zero3}}\ar[u]_{\pi_{U*}}&
Rp_*R\pi_{*}G(n')_{P'}\ar[ur]_{p'_*}\ar[u]^{\pi_{*}} &
\\
Rp_*R\pi_{*} k'_!i'_*G(t)_{X'}\ar[r]^{i'_*}&
Rp_* R\pi_{*}k'_! G(n')_{U'}.\ar[ur]_{\rm nat} & &
}\]
The top left square and the top right triangle commute  by functoriality of the pushforward
(see \cite[Thm 9.7(2)]{BRS}); the square in the middle of the top row commutes since $\pi_{U*}= k^{-1}(\pi_*)$
(compatibility of the pushforward with restriction along smooth maps, see \cite[Thm 9.7(3)]{BRS});
the bottom row of the diagram clearly commutes. 
Thus the whole diagram commutes. 
Noting $Rp_*k_!Rh_* =R\ol{f}_*j_!Rg_* $,
this yields the statement in view of the definition of
$(\ol{f}, j)_*$ and $(\ol{f}', j')_*$ and \eqref{para:ext-zero4}.
\end{proof}

\section{Recollections on {Par\v{s}in} chains, higher local rings, and cohomology}\label{sec:parsin}
In this section we recall some definitions  and results from \cite[(1.6)]{KS-GCFT}.
Let $X$ be a reduced noetherian separated scheme of dimension $d<\infty$, such that $X^{(d)}=X_{(0)}$.
For $U\to X$ \'etale and a subscheme $T\subset X$, we write $U_T=U\times_X T$.

\begin{para}\label{para:chain}
For $x, y\in X$ we write
\[y<x:\Longleftrightarrow \ol{\{y\}}\subsetneq \ol{\{x\}}, \text{ i.e., } y\in \ol{\{x\}} \text{ and } y\neq x.\]
A {\em chain } on $X$ is a sequence
\eq{para:chain1}{\ux=(x_0,\ldots, x_n)\quad \text{with } x_0<x_1<\ldots <x_n.}
The chain $\ux$ is a {\em maximal Par\v{s}in chain} (of just a {\em maximal chain}) if 
$n=d$ and $x_i\in X_{(i)}$. Note that the assumptions on $X$ imply $x_{i}\in \ol{\{x_{i+1}\}}^{(1)}$.
We denote 
\[\c(X)=\{\text{chains on } X\}\quad \text{and} \quad \mc(X)=\{\text{maximal chains on } X\}.\]

Let $0\le r\le d$. A {\em maximal chain with break at $r$} is a chain \eqref{para:chain1}
with $n=d-1$ and $x_i\in X_{(i)}$, for $i<r$, and $x_i\in X_{(i+1)}$, for $i\ge r$.
We denote
\[\mc_r(X)=\{\text{maximal chain with break at $r$ on $X$}\}.\]
For $\ux=(x_0,\ldots, x_{d-1})\in \mc_r(X)$, we denote by $\b(\ux)$ the set of $y\in X_{(r)}$ such that
\eq{para:chain2}{\ux(y):=(x_0,\ldots, x_{r-1}, y, x_{r},\ldots, x_{d-1})\in \mc(X).}
Note that $\b(\ux)$ depends only on $x_r$ and $x_{r-1}$.
\end{para}

\begin{para}\label{para:hlr}
Let $S\subset X$ be a finite subset contained in an affine open neighborhood of $X$.
A {\em strict Nisnevich neighborhood of $S$} is an \'etale map $u:U\to X $ such that $U$ is affine,  
the base change $u^{-1}(S)\to S$ of $u$ is an isomorphism, and every connected component of $U$ intersects 
$u^{-1}(S)$. If $U\to X$ and $V\to X$ are two strict Nisnevich neighborhoods of $S$ there is at most
one $X$-morphism $ U\to V$, by \cite[Expos\'e, I, Corollaire 5.4]{SGA1}.
By taking a representative for each isomorphism class the category of strict Nisnevich neighborhoods of $S$ in $X$
with morphisms the $X$-morphisms, therefore gives rise to a filtered set.

Let $\ux=(x_0,\ldots, x_n)$ be a chain on $X$. A {\em strict Nisnevich neighborhood of $\ux$} 
is a sequence of maps
\[\sU=(U_n\to\ldots \to U_1\to U_0 \to X=:U_{-1}),\]
such that $U_i\to U_{i-1}$ is a strict Nisnevich neighborhood of 
$U_{i-1, x_i}$ for all $i=0,\ldots, n$.
A morphism $\sV\to \sU$ between two strict Nisnevich neighborhoods of $\ux$
consists of a collection of a  $U_{i-1}$-morphism $V_i\to U_i$ for each $i\ge 0$.
There is again at most one such morphism. Given two such neighborhoods $\sU$ and $\sV$ of $\ux$
we can form a new neighborhood $\sU\times_X \sV$.
As above we obtain a filtered set 
\[N(\ux):=\{\text{strict Nisnevich neighborhoods of }\ux\}.\]
Assume $\ux\in \mc_r(X)$ and $y\in \b(\ux)$. Then we have a map of filtered sets
\eq{para:hlr1}{N(\ux)\to N(\ux(y))}
given by 
\[(U_{d-1}\xr{u_{d-1}} \ldots \xr{u_0} X)\mapsto
(U_{d-1}\xr{u_{d-1}} \ldots \xr{u_{r}} U_{r-1} \xr{\id} U_{r-1}\xr{u_{r-1}} \ldots  \xr{u_0} X).\]

Let $F$ be a presheaf of abelian groups on $X_\Nis$ and $\ux=(x_0,\ldots, x_n)\in \c(X)$.
The Nisnevich stalk of $F$ at $\ux$ is defined to be 
\[\Fh_{\ux}:=\varinjlim_{\sU=(U_n\to\ldots\to X)\in N(\ux)} F(U_n).\]
Note that for $\ux\in \mc_r(X)$ and $y\in \b(\ux)$ the map \eqref{para:hlr1}
induces a natural map
\eq{para:hlr2}{\iota_y : \Fh_{\ux}\to \Fh_{\ux(y)}.}
For $\ux=(x_0,\ldots, x_n)\in \c(X)$, write 
$\sOh_{X,\ul{x}}=\Fh_{\ux}$ with $F=\sO$ and $K^h_{X,\ux}=\Frac(\sOh_{X,\ux})$.
By Lemma \ref{lem:hlr} below, $\sOh_{X,\ul{x}}$ is a finite product of henselian local rings and $K^h_{X,\ux}$ is a finite product of fields.
% which are separable algebraic extensions of $\kappa(x_n)=\sO_{X, x_n}$. 
If $x_n\in X_{(d)}$, then $\sOh_{X,\ux}=K^h_{X,\ux}$. Otherwise, we have
\[K^h_{X,\ux}=\underset{\eta\in X_{(d)},\eta>x_n}{\prod} \sOh_{X, (\ux,\eta)}.\]
For $\ux\in \mc_r(X)$ and $y\in \b(\ux)$ we have by \eqref{para:hlr2} natural maps 
\[K^h_{X, \ux}\to K^h_{X, \ux(y)}, \,\text{for }r<d
\quad \text{and} \quad \sO^h_{X, \ux}\to K^h_{X, \ux(y)}={\rm Frac}(\sO^h_{X,\ux}), \,\text{for } r=d.\]
\end{para}

\begin{lem}\label{lem:hlr}
For $\ux=(x_0,\ldots, x_n)\in\c(X)$ we have 
%\begin{enumerate}[label=(\arabic*)]
%\item\label{lem:hlr1} There exists a strict Nisnevich neighborhood $\sU=(U_n\to \ldots \to U_0\to X)\in N(\ux)$ such that 
%for all $i\ge 0$, $U_{i, x_{i-1}}$ is empty and  for all 
%$\varphi: \sV\to \sU$ in $N(\ux)$, the map $\varphi_{i}$ induces a bijection
%\eq{lem:hlr11}{\varphi_{i}: |V_{i, x_i}|\xr{1:1} |U_{i, x_i}|.}
$\sOh_{X,\ux}=R_n$, where 
\[R_0=\sO_{X,x_0}^h \quad \text{and}\quad  \quad R_{i}=\prod_{\fp\in T_i} R_{i-1,\fp}^h,\quad i\ge 1,\]
with $T_i= \Spec R_{i-1}\times_X \{x_i\}$ the finite set of prime ideals in $R_{i-1}$ lying over the prime ideal  in $\sO_{X,x_0}$ corresponding to $x_i$.
%\item\label{lem:hlr3} Let $Y=\ol{\{x_n\}}$ and set $\uy=\ux$ viewed as a chain on $Y$. Then 
%$\sOh_{X,\ux}/\fr = \sOh_{Y,\uy}$, where $\fr$ denotes the radical of $\sOh_{X,\ux}$. 
%\item\label{lem:hlr4} Let $X'$ satisfy the same assumptions as $X$ (the dimension might be different) and let $f:X'\to X$
% be a  map. Let $\ux'\in\c(X')$ and $\ux=f(\ux')$. 
%Then there is a natural map
% $f^* : \sOh_{X,\ux}\to \sOh_{X',\ux'}$; it is finite injective if $f$ is finite surjective. 
%\end{enumerate}
\end{lem}
\begin{proof}
%By \cite[Th\'eor\`eme (18.6.9)]{EGAIV4} (and its proof) we find $\sU$ with $\eqref{lem:hlr11}$.
%Shrinking the $U_i$'s around $U_{i,x_i}$ we obtain $U_{i, x_{i-1}}=\emptyset$.
The finiteness of the $T_i$ holds by \cite[Th\'eor\`eme (18.6.9)]{EGAIV4}.
Set $\ux_i=(x_0,\ldots, x_i)$. As in \eqref{para:hlr1} we obtain maps
$N(\ux_0)\to N(\ux_1)\to \ldots \to N(\ux_n)$. The equality $\sOh_{X,\ux}=R_n$ then follows from 
\[\varinjlim_{\sU\in N(\ux)} \sO(U_n)= 
\varinjlim_{\sV_0\in N(\ux_0)} \varinjlim_{\sV_1\in N(\ux_1)/\sV_0}\ldots
\varinjlim_{\sU\in N(\ux_n)/\sV_{n-1}} \sO(U_n).\]
%For \ref{lem:hlr3} we note that if $\sU\in N(\ux)$, with $\ux\in \c(X)$, then 
%$\sU\times_X Y\in N(\uy)$ and it follows from \cite[Proposition (18.6.8)]{EGAIV4},
%that Nisnevich neighborhoods of the form $\sU\times_X Y$ are cofinal in $N(\uy)$. This implies the statement.
%\ref{lem:hlr4} follows from \ref{lem:hlr2} and the corresponding property of henselization.
\end{proof}

\begin{para}\label{para:hlc}
Let $F$ be an abelian Nisnevich sheaf on $X$. For a finite subset $S\subset X$ we write 
$H^i_S(X_\Nis, F)=\varinjlim_{S\subset V\subset X} H^i_{\ol{S}\cap V}(V_\Nis,F)$,
where the colimit is over all open subsets $V$ of $X$ containing $S$ and $\ol{S}$ denotes the closure of $S$ in $X$.

Let $\ux=(x_0,\ldots, x_n)\in \c(X)$. We set 
\[H^i_{\ux}(X, F):= \varinjlim_{\sU=(U_n\to\ldots\to X)\in N(\ux)} H^i_{U_{n,x_n}}(U_{n,\Nis}, F).\]
Assume $x_{n-1}\in \ol{\{x_n\}}^{(1)}$ and write $\ux'=(x_0,\ldots, x_{n-1})$. We define
a map
\eq{para:hlc1}{ \delta_{\ux}: H^i_{\ux}(X, F)\to H^{i+1}_{\ux'}(X,F)}
as follows: set $y:=x_n$ and $x:=x_{n-1}$. 
For $\sU=(U_n\to\ldots\to U_0\to X)\in N(\ux)$, set $V:=U_n$, $U:=U_{n-1}$, and 
set $U_{(x)}:= U\times_X \Spec\sO_{X,x}$
and define $\delta_{\sU}$ as the composition (we drop the index ``$\Nis$'' everywhere)
\[\delta_{\sU}: H^i_{V_y}(V, F)\cong  H^i_{U_{\ol{\{y\}}}\setminus U_x}(U_{(x)}\setminus U_{x},F)
\xr{\delta} H^{i+1}_{U_{x}}(U_{(x)}, F)\cong H^{i+1}_{U_{x}}(U, F),\]
where the first isomorphism is Nisnevich excision plus $\ol{\{y\}}=\{x,y\}\subset \Spec \sO_{X,x}$,
$\delta$ is the connecting homomorphism from the long exact localization sequence, and the last isomorphism is again
excision.  
Then $\delta_{\ux}=\varinjlim_{\sU\in N(\ux)} \delta_{\sU}$.

Let $\ux=(x_0,\ldots, x_d)\in\mc(X)$. As in \cite[Definition 1.6.2(5)]{KS-GCFT} we define
\eq{para:hlc2}{c_{\ux}:=s_{x_0}\circ c_{\ux,0} : F(K^h_{X,\ux}):= \Fh_{\ux}\to H^d(X_{\Nis}, F),}
where $s_{x_0}: H^d_{x_0}(X_{\Nis}, F)\to H^d(X_{\Nis}, F)$ is the forget-support-map
and $c_{\ux, 0}$ is the composition
\ml{para:hlc3}{
c_{\ux, 0} : \Fh_{\ux}=H^0_{(x_0,\ldots, x_d)}(X_{\Nis}, F)
\xr{\delta_{(x_0,\ldots, x_d)}}H^1_{(x_0,\ldots x_{d-1})}(X_{\Nis}, F)
\xr{\delta_{(x_0,\ldots x_{d-1})}}\ldots \\
\ldots\xr{\delta_{(x_0,x_1)}} H^d_{x_0}(X_{\Nis}, F).}
Note that for $\ux\in \mc_d(X)$, 
%$r\in \{0,\ldots, d\}$, and $a\in \Fh_{\ux}$ we have 
%\eq{para:hlc3.5}{c_{\ux(y)}(\iota_{y}(a))=0, \text{ for almost all }y\in \b(\ux) \quad \text{and}\quad
%\sum_{y\in \b(\ux)}c_{\ux(y)}(\iota_{y}(a))=0,}
%where $\iota_y: \Fh_{\ux} \to \Fh_{\ux(y)}$ is the map \eqref{para:hlr2}.
%This is well-known and in particular contained in \cite[1.6]{KS-GCFT}.
%In case $r=d$, note that 
the composition 
\eq{para:hlc4}{\Fh_{\ux}\xr{\iota_y} \Fh_{\ux(y)}\xr{\delta_{\ux(y)}}
H^1_{\ux}(X_{\Nis}, F)}
is zero, so that $c_{\ux(y)}\circ \iota_y=0$ for all $y\in\b(\ux)$.
\end{para}

\begin{prop}[{\cite[Lemma 1.6.3]{KS-GCFT}}]\label{prop:mapNisCoh}
Let $F$ be an abelian Nisnevich sheaf on $X$. Then for any abelian group $A$ the map
\[\Phi: \Hom(H^d(X_\Nis, F), A)\to \prod_{\ux\in \mc(X)} \Hom(\Fh_{\ux}, A), \quad 
\alpha\mapsto (\alpha\circ c_{\ux})_{\ux\in \mc(X)},\]
is injective. Furthermore, the image of $\Phi$ consists of  those tuples $(\chi_\ux)_{\ux\in\mc(X)}$
satisfying the following condition:
for any $r\in\{0,\ldots, d\}$,  $\ux\in \mc_r(X)$,  $y\in \b(\ux)$ and for any $a\in \Fh_{\ux}$,
we have $\chi_{\ux(y)}(\iota_y (a))=0$ for almost all $y\in \b(\ux)$, and 
\[\sum_{y\in \b(\ux)} \chi_{\ux(y)}(\iota_y(a))=0.\]
%where $\iota_y: \Fh_{\ux}\to \Fh_{\ux(y)}$ is the map from \eqref{para:hlr2}.
\end{prop}
%
%\def\cZar{c^{\Zar}}
%\begin{remark}\label{prop:mapNisCoh;rem1}
%For $w\in X_{(0)}$, the map (cf. \eqref{para:hlc3})
%\[\underset{\ux\in \mc(X),x_0=w}{\bigoplus} \Fh_{\ux} \overset{c_{\ux,0}}{\longrightarrow}
%H^d_w(X_\Nis, F)\]
%is surjective, where the sum is over all $\ux=(x_0,\dots,x_d)\in \mc(X)$ such that 
%$x_0=w$. This follows from Proposition \ref{prop:mapNisCoh}
%applied to $\Spec(\sO_{X,w})\backslash \{w\}$.
%\end{remark}
%
%\begin{remark}\label{prop:mapNisCoh;rem2}
%Let $F$ be a presheaf of abelian groups on $X_\Zar$ and $\ux=(x_0,\ldots, x_n)\in \c(X)$.
%By the same argument, we can define the Zariski stalk $F_{\ux}$ of $F$ at $\ux$ but 
%$F_{\ux}=F_{x_n}$.
%If $\ux\in \mc(X)$, we can define also the map analogous to \eqref{para:hlc3}:
%\eq{para:hlc3Zar}{
%\cZar_{\ux, 0} : F_{\ux} \to H^d_{x_0}(X_{\Nis}, F).}
%The Zariski analogue of Proposition \ref{prop:mapNisCoh} holds. In particular,
%for $w\in X_{(0)}$, the map
%\[\underset{\ux\in \mc(X),x_0=w}{\bigoplus} F_{\ux} \overset{\cZar_{\ux,0}}{\longrightarrow}
%H^d_w(X_\Zar, F)\]
%is surjective.
%\end{remark}

\section{Reciprocity Pairings and Zariski-Nagata Purity}
In this section $k$ is perfect. We fix $F\in\RSC_\Nis$ and set
 $\tF=\tau_!\omega^{\CI}F\in \CItspNis$, see \eqref{para:RSC1.5}.
 
 \begin{defn}\label{defn:Fgen}
Let $X$ be a reduced separated scheme of finite type over a function field $K$ over $k$
 and let $D\subset X$ be a closed subscheme,
such that $U=X\setminus D$ is regular (in this case $U$ is also pro-smooth over $k$).
We define
\[F_\gen(X,D):= \Ker\left(F(U)\to \bigoplus_{x\in \tX^{(1)}\cap \nu^{-1}(D)} 
\frac{F(\tX_{x}^h \setminus x)}{\tF(\tX_{x}^h, D^h_{x})}\right),\]
where $\nu :\tilde{X}\to X$ is the normalization, $\tX_{x}^h=\Spec \sO_{\tX,x}^h$, and 
$D^h_{x}= D\times_\tX \tX_{x}^h$.
\end{defn}
The aim of this section is to show that if $X$ is smooth and projective over $K$ and $D_\red$ is a SNCD on $X$,
then $F_{\gen}(X, D)=\tF(X,D)$, i.e., the modulus of an element $a\in F(U)$ is determined by 
the local modulus (or the motivic conductor in the language of \cite{RS}) at the 1-codimensional points of $X$.
If the support of $D$ is smooth, this also follows from \cite[Corollary 8.6(2)]{S-purity} and if $D$ is reduced, it follows
from \cite[Theorem 2.3]{SaitologRSC}. The proof of the  general case given here is completely different.

\begin{lem}\label{lem:Fgen-fp}
Let $X$ and $Y$ be separated and finite type schemes over (maybe different) function fields over $k$ and let 
$f:Y\to X$ be a {\em flat} morphism. Let $D\subset X$ be a closed subscheme such that $U=X\setminus D$ and $V=Y\setminus f^{-1}(D)$ are regular.
Then the pullback $f^*: F(U)\to F(V)$ induces a morphism
\[f^*: F_{\gen}(X,D)\to F_{\gen}(Y,f^{-1}(D)).\]
\end{lem}
\begin{proof}
Let $y\in Y^{(1)}$ and $x=f(y)$. Since $f$ is flat we have $\dim \sO_{Y,y}= \dim\sO_{X,x} + \dim (\sO_{Y,y}/\fm_x\sO_{Y,y})$.
Thus a 1-codimensional point in $Y$ can only be mapped  to a $0$- or $1$-codimensional point in $X$. The statement follows.
\end{proof}

\begin{para}\label{para:KM}
Let $X$ be a $k$-scheme. We denote by $K^M_{r, X}$ ($r \ge 0$) the Nisnevich sheafification 
of the improved Milnor K-theory from \cite{Kerz}. Let $D\subset X$ be a closed subscheme  and denote by 
$j: U:=X\setminus D\inj  X$  (resp. $i:D\inj  X$) the corresponding open (resp. closed) immersion. We consider the following Nisnevich sheaves for $r\ge 1$
\[V_{r,X|D}:=\Im(\sO_{X|D}^\times\otimes_{\Z}K^M_{r-1, X}\to K^M_{r,X})
\;\;\text{ with } \sO_{X|D}^\times:= \Ker(\sO_X^\times\to i_*\sO_D^\times).\]
Note $V_{1, X|D}=\sO_{X|D}^\times$.
Let $I\subset \sO_X$ denote the ideal sheaf of $D\inj X$. 
In \cite[(1.3)]{KS-GCFT} the sheaf $K^M_{r}(\sO_X, I)=\Ker(K^M_{r, X}\to K^M_{r, D})$ is considered.
The inclusion  $V_{r, X|D}\subset K^M_r(\sO_X,I)$  is an equality for $r=1$ and for $r\ge 2$ it induces an equality between 
the Nisnevich stalks at points $x\in X$ with infinite residue field
(this follows from \cite[Lemma 1.3.1]{KS-GCFT} and \cite[Proposition 10(5)]{Kerz}).
In particular, if $X$ has dimension $d$, Grothendieck-Nisnevich vanishing implies that the natural map
\eq{para:KM1}{H^d(X_{\Nis}, V_{d, X|D})\xr{\simeq} H^d(X_{\Nis}, K^M_d(\sO_X, I))}
is an isomorphism.
Now assume $X$ is of finite type and pure dimension $d$ over $k$.
Let $U\subset X$ be a regular dense open subscheme. 
For $x\in U_{(0)}$ the Gersten resolution (\cite[Proposition 10(8)]{Kerz})  yields an isomorphism 
\eq{para:CHM2}{\theta_x: \Z\xr{\simeq} H^d_x(U_{\Nis}, K^M_{d, U})
\cong H^d_x(X_{\Nis}, V_{d, X|D}).}
By \cite[Theorem 2.5]{KS-GCFT} and \eqref{para:KM1}, we obtain a surjective map
\eq{para:CHM3}{\theta=\sum_x \theta_x: Z_0(U)=\bigoplus_{x\in U_{(0)}} \Z
\surj H^d(X_{\Nis}, V_{d, X|D}).}

\begin{prop}[\cite{RSCycleMap}]\label{prop:cyle-map}
Let the notation be as above and assume additionally that  $X$ is smooth of pure dimension $d$ and that the support of $D$ is a simple normal crossing divisor.
Then $\eqref{para:CHM3}$ factors to give a surjective  map
\[\CH_0(X|D)\surj H^d(X_{\Nis}, V_{d, X|D}).\]
Here, $\CH_0(X|D)$ is the Chow group of zero-cycles with modulus introduced in \cite{Kerz-Saito}, see the introduction.
\end{prop}
\end{para}

%\begin{para}\label{para:CHM}
%Let $(X,D)\in \ulMCor$ with $U=X\setminus |D|$. For $C\subset X$ an integral curve  not contained in the support of $D$ and  with normalization $\nu:\tC\to C$, we set
%\begin{equation}\label{para:CHM1}
%G(C,D):= \bigcap_{x\in \tC\cap \nu^{-1}(D)} \Ker\big(\sO_{\tC,x}^\times \to \sO_{\tC\times_C D,x}^\times\big)\;\subset k(C)^\times.
%\end{equation}
%We define a map
%\[  \partial_C: G(C, D) \rmapo{\div_{\tC}} Z_0(\tC) \rmapo{\nu_*} Z_0(U),\]
%and put
%\[ \CH_0(X|D)=\Coker\left(\partial=\sum_C\partial_C : \underset{C}{\bigoplus}\;G(C,D) \to Z_0(U)\right),\]
%where the sum is over all $C\subset X$ as above and $Z_0(Y)$ denotes the group of zero-cycles on $Y$.
%\end{para}

\begin{para}\label{para:pairing}
Let $U\to S$ be  a quasi-projective dominant morphism in $\Sm$, with  $U$ and $S$ integral.
Set $d=\dim U-\dim S$.
We fix a compactification
\[U\xr{j} X\xr{f} S\]
with $j$ a dense open immersion, $f$ projective, and $X$ integral.
There is a natural map of abelian Nisnevich sheaves on $U$ (see \ref{para:twist}\ref{para:twist3.3})
\eq{para:pairing0}{F_U\otimes_\Z K^M_{d,U}\xr{\rm nat.} (F\otimes_{\NST} K^M_d)_U\to F\la d\ra_U}
inducing
\[K^M_{d,U}\to \uHom_{{\rm Sh}(U_\Nis)}(F_{U}, F\la d\ra_{U})= j^{-1}\uHom_{{\rm Sh}(X_\Nis)}(j_*F_U, j_!F\la d\ra_U).\]
By adjunction we obtain the following morphism in ${\rm Sh}(X_\Nis)$
\eq{para:pairing1}{j_*F_U\otimes_\Z j_!K^M_{d,U}\to j_!F\la d\ra_U.}
We define  the pairing 
\eq{para:pairing2}{(-,-)_{U\subset X/S}:F(U)\otimes_\Z H^d(X_\Nis, j_!K^M_{d,U})\to F(S)}
as the composition:
\mlnl{F(U)\otimes_\Z H^d(X_\Nis, j_!K^M_{d,U})\to H^d(X_\Nis, j_*F_U\otimes_\Z j_!K^M_{d,U})\\
\xr{\eqref{para:pairing1}} H^d(X_\Nis, j_!F\la d\ra_U)\xr{(f,j)_*} F(S),}
where $(f,j)_*$ is the pushforward from Definition \ref{defn:pfcs} and
the first map sends $a\in F(U)$ and $\beta\in H^d(X_\Nis, j_!K^M_{d,U})$ corresponding to
maps $a:\Z_X\to j_*F_U$ and $\beta: \Z_X\to j_!K^M_{d,U}[d]$ in $D(X_\Nis)$, respectively,
to $a\otimes \beta: \Z_X\to (j_*F_U\otimes_{\Z} j_!K^M_{d,U})[d]$. 
\end{para}

\begin{lem}\label{lem:pair-cor}
Let the assumptions be as above. Let $Z\subset U$ be an integral closed subscheme which is finite and surjective over
$S$. By abuse of notation, we also denote by $Z$ the finite correspondence in $\Cor(S,U)$ defined by the image of the
natural map $Z\to S\times U$ and we denote by $[Z]$ the image of $1\in \Z$ under the following map
\[\Z\xr{\theta_Z,\, \simeq} H^d_{Z}(U_\Nis, K^M_{d,U})=H^d_{Z}(X_\Nis, j_!K^M_{d,U})\to
H^d(X_\Nis, j_!K^M_{d,U}),\] 
where the isomorphism $\theta_Z$ is induced by the Gersten resolution, noting $\codim_U(Z)=d$.
Then for $a\in F(U)$
\[(a, [Z])_{U\subset X/S}= Z^*a \quad \text{in } F(S).\]
If $Z$ is furthermore smooth  and we denote by $i: Z\inj U$ the closed immersion and $g: Z\to S$ the finite 
and surjective map, then
\eq{lem:pair-cor1}{(a, [Z])_{U\subset X/S}=g_*i^*a,}
with $g_*$ as in \eqref{para:proj-pf1}.
\end{lem}
\begin{proof}
Since the restriction map  $F(S)\to F(V)$ is injective for a dense open $V\subset S$ (see \cite[Theorem 3.1]{S-purity}), 
we can shrink $S$ around its generic point and therefore may assume that $Z$ is smooth.
Let $\Gamma_i\in \Cor(Z, U)$ and $\Gamma_g^t\in \Cor(S, Z)$ be  the (transpose of the) graph of 
$i$ and $g$ from the second part of the statement, respectively.
We have 
\[Z^*a= (\Gamma_i\circ \Gamma_g^t)^* a= (\Gamma_g^t)^*\Gamma_i^* a= g_*i^*a,\]
where for the last equality we use that 
$g_*: F(Z)\to F(S)$ from \eqref{para:proj-pf1} is equal to $(\Gamma_g^t)^*$ by \cite[Proposition 8.10(3)]{BRS}.
Thus it remains  to show \eqref{lem:pair-cor1}. To this end, consider the diagram
\[\xymatrix{
F(U) \ar[r]^-{\cup \theta_Z(1)} \ar[d]_{i^*} & H^d_Z(U_\Nis,F\la d\ra) \ar[d]\\
F(Z) \ar[ru]^-{i_{Z_*}} \ar[rd]_-{g_*} &  H^d(X_\Nis,j_! F\la d\ra_U )\ar[d]^{(f,j)_*}\\
& F(S)}\]
where $i_{Z_*}$ is induced by $H^d_Z$ of the map \eqref{para:proj-pf1}, where
$f=i$, $r=-d$, $a=d$, $b=0$, 
and $\cup \theta_Z(1)$ is induced by the pairing
\[ F(U) \times H^d_Z(U_\Nis,K^M_{d,U}) \to H^d_Z(U_\Nis,F\la d\ra_U).\]
Since by definition $i_{Z_*}$ is induced by the Gysin map from \cite[7.4]{BRS},
the top triangle commutes by Theorem 7.11 of {\em loc. cit.}
The lower triangle is commutative by Lemma \ref{lem:pfcs}\ref{lem:pfcs3}.
The equality \eqref{lem:pair-cor1} follows from this.
\end{proof}

\begin{prop}\label{prop:mod-K-pairing}
Let $K$ be a function field over $k$. 
Let $X$ be a reduced and projective  $K$-scheme of pure dimension $d$ and $D\subset X$ a nowhere dense 
closed subscheme, such that $U=X\setminus D$ is regular. Denote by $j:U\inj X$ the open immersion. Then the natural map
$H^d(X_\Nis, j_!K^M_{d,U})\to H^d(X_\Nis, V_{d, X|D})$ is surjective (see \ref{para:KM} for notation)
and there is a unique pairing 
\eq{prop:mod-K-pairing1}{(-,-)_{(X,D)/K}:F_\gen(X,D)\otimes H^d(X_\Nis, V_{d, X|D})\to F(K)}
such that the following diagram commutes
\[\xymatrix{
F(U) \ar[r]^-{\eqref{para:pairing2}} &\Hom( H^d(X_\Nis, j_!K^M_{d, U}),  F(K))\\
F_\gen(X,D)\ar@{^(->}[u]\ar[r]^-{\eqref{prop:mod-K-pairing1}} & \Hom( H^d(X_\Nis, V_{d, X|D}),F(K)).\ar@{^(->}[u]
}\]
\end{prop}
\begin{proof}
Since $V_{d, X|D}/j_!K^M_{d,U}$ has support in $|D|$, we see
$H^d(X_\Nis, V_{d, X|D}/j_!K^M_{d,U})=0$, whence the surjectivity of the natural map
\eq{prop:mod-K-pairing2}{H^d(X_\Nis, j_!K^M_{d,U})\surj H^d(X_\Nis, V_{d, X|D}).}

To show the existence of the pairing \eqref{prop:mod-K-pairing1}, we first reduce to the case where $X$ is normal.
Let $\nu:\tX\to X$ be the normalization and denote by $\tilde{j}:U\inj \tX$ the induced open immersion. 
We obtain a diagram of solid arrows
\[\xymatrix{
F_\gen(X,D)\ar@{=}[d]\ar[r] &
H^d(X_\Nis, j_!K^M_{d, U})^\vee\ar[d]^{\simeq}&
H^d(X_\Nis, V_{d, X|D})^\vee\ar@{_(->}[l]\\
F_\gen(\tilde{X},\nu^{-1}D)\ar[r]\ar@{-->}@/_1.5pc/[rr]&
H^d(\tX_\Nis, \tilde{j}_!K^M_{d, U})^\vee&
H^d(\tX_\Nis, V_{d, \tX|\nu^{-1}D})^\vee\ar@{_(->}[l]\ar[u]^{(\nu^*)^\vee}
}\]
where $(-)^\vee=\Hom( -,  F(K))$, the top and bottom horizontal maps on the left are 
induced by the pairing \eqref{para:pairing2},
the horizontal injections on the right by \eqref{prop:mod-K-pairing2}, and the vertical isomorphism in the middle by
$j_!\cong \nu_*\tilde{j}_!$, see \eqref{para:ext-zero2}. It is direct to check that the diagram commutes.
Assume the statement is proven for $(\tX, \nu^{-1}D)$. Then  the dashed arrow in the diagram exists, which 
induces  the desired pairing for $(X,D)$. 

  From now on we assume that $X$ is normal and integral. Let $a\in F_\gen(X,D)\subset F(U)$ and denote
by $\chi_a\in \Hom(H^d(X_\Nis, j_!K^M_{d, U}), F(K))$ the image of $a$ under the pairing  \eqref{para:pairing2}.
For $\ux\in\mc(X)$, set (see \eqref{para:hlr} for notation)
\[\chi_{a,\ux}:=\chi_a\circ c_{\ux}\in\Hom(K^M_d(K^h_{X,\ux}), F(K)) ,\]
where $c_\ux$ is the map \eqref{para:hlc2}. 
Note that for all $\ux\in \mc(X)\cup \bigcup_{r=0}^{d-1}\mc_r(X)$,
we have 
\[(V_{d, X|D})^h_{\ux}= (j_!K^M_{d,U})^h_{\ux}= K^M_d(K^h_{X,\ux}).\]
Thus by Proposition \ref{prop:mapNisCoh} the map $\chi_a$  lies in the image of
\[\Hom( H^d(X_\Nis, V_{d, X|D}),F(K))\inj \Hom( H^d(X_\Nis, j_!K^M_{d, U}),  F(K))\]
if  the following composition vanishes for all $\ux\in \mc_d(X)$:
\eq{prop:mod-K-pairing3}{
(V_{d, X|D})^h_{\ux}\xr{\iota_\eta \; \eqref{para:hlr2}} (V_{d, X|D})^h_{\ux(\eta)}=K^M_d(K^h_{X,\ux(\eta)}) \xr{\chi_{a, \ux(\eta)}} F(K), }
where $\eta\in X$ is the generic point.
Let $\ux\in\mc_d(X)$ and set $\uy:=\ux(\eta)$.
By the normality of $X$ and Lemma \ref{lem:hlr}, the ring $\sOh_{X,\ux}$ is a finite product of henselian dvr's of geometric type over $k$ and 
$K^h_{X,\uy}=\Frac(\sOh_{X,\ux})$.
Let $a_{\uy}\in F(K^h_{X,\uy})$ be the restriction of $a$.
By definition the map $\chi_{a,\uy}$ is equal to  the composition
\eq{prop:mod-K-pairing4}{
 K^M_d(K^h_{X,\uy}) \rmapo{\cup a_{\uy}} F\la d\ra(K^h_{X,\uy}) \xr{c_\uy} 
  H^d(X_\Nis,j_! F\la d\ra_U) \rmapo{(f,j)_*} F(K),}
where the first map is induced by the restriction $a_{\uy}$ of $a$ via the pairing \eqref{para:pairing1}
\[ F(K^h_{X,\uy})\otimes K^M_d(K^h_{X,\uy}) \to F\la d\ra(K^h_{X,\uy}).\]
In fact in this description of $\chi_{a,\uy}$, we use that the following diagram commutes
\eq{prop:mod-K-pairing4a}{
\xymatrix{
F(U)\otimes K^M_d(K^h_{X,\uy})\ar[r]^-{\id\otimes c_{\uy}}\ar[d]^{\cup} &
H^0(X_\Nis, j_*F_{U})\otimes H^d(X_{\Nis}, j_!K^M_{d,U})\ar[d]^{\cup}\\
F\la d\ra(K^h_{X,\uy})\ar[r]^-{c_{\uy}} & H^d(X_{\Nis}, j_!F\la d\ra_U) } }
which follows from the compatibility of the cup product with the boundary maps in long exact sequences, e.g., \cite[Corollary 3.7]{Swan}. 
Since $X$ is normal, we find  a regular open subset $V\subset X$  which contains $U$ and $X^{(1)}$;
let $j':V\to X$ be the open immersion. 
By Lemma \ref{lem:pfcs}\ref{lem:pfcs2}, we have a commutative diagram
\[\xymatrix{
F(d)(K^h_{X,\uy})\ar[r]^-{c_{\uy}} \ar[rd]_-{c_{\uy}} & H^d(X_\Nis,j_! F \la d \ra_U)
\ar[r]^-{(f,j)_*}\ar[d] & F(K)\\
&H^d(X_\Nis,j'_! F \la d\ra_V). \ar[ru]_-{(f,j')_*}
\\}\]
Since $X^{(1)}\subset V$ we have $(j'_! F\la d\ra_V)_{\ux}=F\la d \ra(\sOh_{X,\ux})$, 
and hence the lower $c_{\uy}$ annihilates 
$F\la d\ra(\sOh_{X,\ux})$, see  \eqref{para:hlc4}.
Thus to show  the vanishing of \eqref{prop:mod-K-pairing3}, it suffices to show
that the image of the composition
\[F_{\gen}(\sOh_{X,\ux}, I^{-1}_\ux) \otimes V_{d, X|D,\ux}\to F(K^h_{X,\uy}) \otimes K^M_d(K^h_{X,\uy})\xr{\eqref{para:pairing0}}
F\la d  \ra(K^h_{X,\uy})\]
lies in $F\la d\ra(\sOh_{X,\ux})$, where $I_\ux\subset\sOh_{X,\ux}$ denotes the ideal of $D$ around $\ux$ and 
we use the notation from \eqref{para:RSC2}.
Since $F_\gen(\sOh_{X,\ux}, I^{-1}_\ux)=\tF(\sOh_{X,\ux}, I^{-1}_\ux)$ by definition,
the image of the above composition is equal to the image of the following composition
\[(\tF(\sOh_{X,\ux}, I^{-1}_\ux)\otimes \sO_{X|D,\ux}^\times) \otimes K^M_{d-1}(\sOh_{X,\ux})\to F\la1\ra(K^h_{X,\uy})\otimes K^M_{d-1}(\sOh_{X,\ux})
\to F\la d  \ra(K^h_{X,\uy}).\]
Hence  the desired assertion follows from Theorem \ref{thm:cont}.
\end{proof}

%The Zariski-Nagata purity theorem states that the branch locus  of a generically \'etale and quasi-finite morphism
%$X\to Y$ between smooth integral $k$-schemes has pure dimension 1. 
%This translates into the fact that if $V\subset X$ contains $X^{(1)}$,
%then the `etale fundamental groups of $V$ and $X$ agree, see \cite[Exp X, Corollaire 3.3]{SGA1}.
%In particular, setting  $F(X)=\Hom_{\rm cts}(\pi_1(X)^{\rm ab}, \Q/\Z)$
%we have $F(X)=F(V)$. In  this sense the theorem below  is a version of this purity result
%for reciprocity sheaves 

\begin{thm}\label{thm:ZNP}
Let $F\in \RSCNis$. Let $X$ be a smooth projective $k$-scheme of pure dimension $d$ and $D$ an effective Cartier divisor on $X$, such that  its support $|D|$ is SNCD.
Denote by $j: U=X\setminus|D|\inj X$ the open immersion.
For $a\in F(U)$, the following conditions are equivalent:
\begin{enumerate}[label=(\roman*)]
\item\label{thm:ZNP1}
$a\in \tF(X,D)$;
\item\label{thm:ZNP2}
$a\in F_{\gen}(X,D)$;
\item\label{thm:ZNP3}
for any function field $K$ over $k$, the map
\[(a_K,-)_{U_K\subset X_K/K} : H^d(X_{K,\Nis},j_!K^M_{d,U_K}) \to F(K)\]
induced by the pairing \eqref{para:pairing2} (with $a_K\in F(U_K)$ the  pullback of $a$)   factors through $H^d(X_{K,\Nis}, V_{d, X_K|D_K})$.
\end{enumerate}
\end{thm}
\begin{proof}
The implication \ref{thm:ZNP1}$\Rightarrow$\ref{thm:ZNP2} is obvious. 
If $a\in F_\gen(X,D)$, then Lemma \ref{lem:Fgen-fp} implies $a_K\in F_\gen(X_K, D_K)$. 
Thus \ref{thm:ZNP2}$\Rightarrow$\ref{thm:ZNP3} follows from Proposition \ref{prop:mod-K-pairing}.
Assume \ref{thm:ZNP3} holds. We have to show that this implies \ref{thm:ZNP1}, i.e.,
that the Yoneda map $a: \Ztr(U)\to F$ factors through the quotient map
$q: \Ztr(U)\to h_0(X,D)$ in $\PST$, see \ref{para:RSC}.
This means that we have to show that $a(S): \Ztr(U)(S)\to F(S)$ factors through $q(S)$, for any $S\in \Sm$.
If $S$ is connected with function field $K$, then $F(S)\to F(K)$ is injective by \cite[Theorem 3.1]{S-purity}.
Hence it suffices to show the claim in case $S=\Spec K$.
By \cite[Theorem 3.3]{BS19} we have a commutative diagram in which the vertical maps are canonical isomorphisms
\eq{thm:ZNP4}{\xymatrix{
\Ztr(U)(K)\ar[r]^{q(K)}\ar@{=}[d] & h_0(X,D)(K)\ar[d]^{\simeq}\\
Z_0(U_K)\ar[r]^-{\pi}  &\CH_0(X_K|D_K).
}}
where $\CH_0(X_K|D_K)$ is the Chow group of zero-cycles with modulus introduced in the introduction.
Thus we are reduced to show the following.

\begin{claim}\label{thm:ZNP-claim}
%Putting $G=\Ker(q)$, $a(S): \Ztr(U)(S)\to F(S)$ annihilates $G(S)$ for any $S\in \Sm$.
$a(K): Z_0(U_K)=\Ztr(U)(K)\to F(K)$ factors via $\pi$ from \eqref{thm:ZNP4}.
%$Z_0(U_K)\to \CH_0(X_K|D_K)$. for $K$ a function field over $k$.
\end{claim}
By Lemma \ref{lem:pair-cor} we have a commutative diagram
\[\xymatrix{
F(U) \ar[r]^-{\rho_K}\ar[d]_{\simeq} &\Hom(H^d(X_{K,\Nis},j_!K^M_{d,U_K})),F(K)) \ar[d]_-{\gamma} \\
\Hom_{\PST}(\Ztr(U),F) \ar[r]^-{\beta} & \Hom(\Ztr(U)(K),F(K)) \\}\]
where $\rho_K$ is induced by \eqref{para:pairing2}, $\beta$ is the evaluation map, and $\gamma$ is induced by
\eq{thm:ZNP5}{\Ztr(U)(K)=\underset{z\in (U_K)_{(0)}}{\bigoplus}\Z \xrightarrow[\simeq]{\eqref{para:CHM2}}
\bigoplus_{z\in (U_K)_{(0)}} H^d_z(U_{K,\Nis}, K^M_{d,U})\xr{\sum} 
H^d(X_{K,\Nis},j_!K^M_{d,U_K}).}
By the diagram we have $a(K)=\gamma(\rho_K(a))=(a_K,-)_{U_K\subset X_K/K}\circ\eqref{thm:ZNP5}$.
Moreover, by \eqref{thm:ZNP4} and Proposition \ref{prop:cyle-map}, the composite of \eqref{thm:ZNP5} 
with the natural map
\[H^d(X_{K,\Nis},j_!K^M_{d,U_K}))\to H^d(X_{K,\Nis},V_{d, X_K|D_K})\]
factors as
\[  \Ztr(U)(K)\rmapo{q(K)} h_0(X,D)(K)\simeq\CH_0(X_K|D_K) 
\to H^d(X_{K,\Nis},V_{d, X_K|D_K}).\]
Thus the condition \ref{thm:ZNP3} implies Claim \ref{thm:ZNP-claim}. This completes the proof.
\end{proof}

\begin{exs-rmks}\label{rmk:CFT}
\hspace{0.1cm}
\begin{enumerate}[label=(\arabic*)]
\item Let $F$ be the reciprocity sheaf given by 
\[F(X)=\Hom_{\rm cts}(\pi_1^{\et}(X), \Q/\Z), \quad X\in\Sm,\]
see \cite[8]{RS} for the fact that this is a reciprocity sheaf.
Set $\pi_1^{\et,{\rm ab}}(X,D):=\Hom(\tF(X,D),\Q/\Z)$, cf. \cite[Definition 2.9]{Kerz-Saito}.
Let $(X,D)$ be as in Theorem \ref{thm:ZNP} and assume that $k$ is a finite field.
Then $F(k)\cong \Q/\Z$  and the pairing $(-,-)_{(X,D)/K}$ induces a morphism
\[H^d(X_\Nis, V_{d, X|D})\to \pi_1^{\et, {\rm ab}}(X,D).\]
Taking the inverse limit over multiples of $D$ we obtain the reciprocity map constructed in
\cite[(3.7)]{KS-GCFT}. (Actually in {\em loc. cit.} the case where $k$ is a number field is considered, which requires some
 additional attention to the places at infinity.) Composing with the cycle map
from Proposition \ref{prop:cyle-map} yields the reciprocity map 
\[\CH_0(X|D)\to \pi_1^{\et, {\rm ab}}(X,D),\]
from \cite[Proposition 3.2]{Kerz-Saito}. (But in {\em loc. cit.} only $X\setminus |D|$ is assumed to be smooth.)
Finally we note that in the special case at hand and in view of \cite[Theorem 7.16]{Gupta-Krishna-Reciprocity}, 
Proposition \ref{prop:mod-K-pairing} essentially reproves \cite[Theorem 1.2]{Gupta-Krishna-Reciprocity}.
In fact, in {\em loc. cit.} $X$ is just assumed to be of finite type instead of being projective. However,  
if $X$ is at least quasi-projective we can construct a pairing as in Proposition \ref{prop:mod-K-pairing}
by considering projective compactifications of $(X,D)$. 
Also note that the paring in {\em loc. cit.} is constructed only over a finite field.

\item Assume ${\rm char}(k)=0$ and let $(X,D)$ be as in Theorem \ref{thm:ZNP}.
Denote by ${\rm Conn}^1_{\rm abs}(X)$ 
the group of isomorphism classes of absolute  rank one connections on $X$ relative to $k$.
(``Absolute'' refers to the fact that the connection has values in absolute differentials, 
$\nabla:\sL\to \sL\otimes_{\sO_X}\Omega^1_{X/\Z}$.)
This defines a reciprocity sheaf, cf. \cite[6.10]{RS}. 
Theorem \ref{thm:ZNP-claim} yields a pairing 
\[\widetilde{{\rm Conn}}^{1}_{\rm  abs}(X,D)\otimes H^d(X_\Nis, V_{d,X|D})\to {\rm Conn}^{1}_{\rm abs}(k).\]
By construction and (an absolute version of) \cite[Theorem 6.11]{RS}, this pairing 
is a higher-dimensional version of  the pairing constructed  in \cite[4.]{Bloch-Esnault} in the case $X$ is a curve, i.e., $d=1$.
\item\label{rmk:CFT3}  Assume ${\rm char}(k)=p>0$. 
For $j\ge 1$ and $X\in \Sm$ consider
\[H^j_{p^n}(X):= H^0(X, R^j\e_*\Z/p^n(j-1))= H^0(X, R^1\e_*W_n\Omega^{j-1}_{X,\log}),\]
where $\e: X_{\et}\to X_\Nis$ is the natural change of sites morphism, $W_n\Omega^{j-1}_{X,\log}$ denotes the 
subsheaf of the de Rham-Witt differentials $W_n\Omega^{j-1}_X$ generated locally by log-forms (see \cite{Il}), and 
$\Z/p^n(j-1)$ denotes the mod-$p^n$-motivic complex of weight $j-1$.
The last equality above holds by \cite[Theorem 8.3]{GL}. In fact we have 
\[H^j_{p^n}= {\rm Coker}(W_n\Omega^{j-1}\xr{F-R} W_n\Omega^{j-1}/dV^{n-1}\Omega^{j-1}).\]
Since $\RSC_{\Nis}$ is an abelian category by \cite[Theorem 0.1]{S-purity} it follows that $H^j_{p^n}$ is a reciprocity sheaf.

Let  $(X,D)$ be as in Theorem \ref{thm:ZNP}. 
By \cite{Iz} (see also \cite[Theorem 8.1]{GL}) and the Gersten resolution \cite[Proposition 10(8)]{Kerz}, $K^M_{d, X}$ is $p$-torsion free, hence so is $V_{d, X|D}$. Grothendieck-Nisnevich vanishing yields 
$H^d(X_\Nis, V_{d, X|D})/p^n= H^d(X_\Nis, V_{d, X|D}/p^n)$.
Thus for every function field $K/k$, Theorem \ref{thm:ZNP} yields a pairing
\[\widetilde{H^j_{p^n}}(X_K,D_K)\otimes H^d(X_{K,\Nis}, V_{d, X_K|D_K}/p^n)\to H^j_{p^n}(K),\]
which  induces a pairing, for all $j\ge 1$
\[H^j_{p^n}(U_K)\times \varprojlim_{r} H^d(X_{K,\Nis}, V_{d, X_K|r D_K}/p^n)\to H^j_{p^n}(K).\]
This should be compared to  \cite[Theorem 2]{JSZ} and 
\cite[Theorem 1.1]{Gupta-KrishnaII} where similar pairings are constructed in different cohomological degrees.
Note that the  filtrations on $H^j(U_{\et}, \Z/p^n(j-1))$  used there to define the pairings 
are more ad hoc and  only well-behaved in the colimit.
\end{enumerate}
\end{exs-rmks}

%\begin{cor}\label{cor:ZNP}
%Let $(X,D)\in \ulMCor$ and assume that there is a dense open immersion $j:X\hookrightarrow \ol{X}$ 
%and a divisor $D_1$ on $\ol{X}$
%such that $\ol{X}$ is proper and smooth, $|D_1|$ is SNCD,  $D_{1|X}=D$, and $X\setminus|D|=\ol{X}\setminus |D_1|$. 
%Then
%\[\tF(X,D)=F_\gen(X,D).\]
%\end{cor}
%\begin{proof}
%Set $U=X\setminus|D|$. Write $D_1= \ol{D}+ H$, where $\ol{D}$ denotes the closure of $D$ in $D_1$ 
%(viewed as closed subscheme).
%Take $a\in F_\gen(X,D)\subset F(U)$. 
%For a sufficiently large $n>0$, we have
%\eq{cor:ZNP1}{a\in \tF(\ol{X}, n(\ol{D}+H))\subset  F_\gen(\ol{X}, n(\ol{D}+ H)).}
%By  assumption 
%\[ a\in \Ker\left( F(U) \to \underset{\eta\in D^{(0)}}{\bigoplus}
%\frac{F(X_{\eta}^h-\eta)}{F(X_{\eta}^h, D_\eta^h)}\right).\]
%Combined with \eqref{cor:ZNP1}, this implies
%\[ a\in F_\gen(\ol{X}, \ol{D}+ n H)= \tF(\ol{X}, \ol{D}+ n H),\]
%where the equality follows from Theorem \ref{thm:ZNP}.
%Hence  $a\in \tF(X,D)$, as desired.
%\end{proof}

%\begin{para}\label{para:level}

%\end{para}

By \cite[Corollary 8.6]{S-purity} we have 
$\tF(X,D)=F_{\gen}(X,D)$, 
for any  $(X,D)\in\uMCor$ with $X$ and $|D|\in \Sm$. 
The following corollary of Theorem \ref{thm:ZNP} generalizes
this result to the case where $|D|$ is only SNCD - at least under a mild extra assumption.
Recall from \cite[Definition 1.8.1]{KMSY1}, that a {\em compactification of a modulus pair} $(X,D)\in\uMCor$
is a proper modulus pair $(\ol{X}, \ol{D}+B)\in \MCor$ with effective Cartier divisors $\ol{D}$ and $B$,
such that there is a dense open immersion $j: X\inj\ol{X}$ with $j(X)=\ol{X}\setminus |B|$ and $D=j^*\ol{D}$.
A compactification of $(X,D)$ always exists.

\begin{cor}\label{cor:cut-by-curve}
Assume $F$ has  level $n\ge 0$ (see Definition \ref{def;RSCNis}) and 
resolutions of singularities hold over $k$ in dimension $\le n$ (see Theorem \ref{intro:thm1}).
Let $(X,D)\in\uMCor$  and  $U=X\setminus |D|$. Let $a\in F(U)$.
The following statements are equivalent:
\begin{enumerate}[label=(\roman*)]
\item\label{cor:cut-by-curve1} $a\in \tF(X,D)$;
\item\label{cor:cut-by-curve2} $h^*a\in \tF(Z, h^*D)$, for all $k$-morphisms 
$h:Z\to X$ with $\dim(Z)\le n$, such that $h^{-1}(U)$ is smooth and non-empty;
\item\label{cor:cut-by-curve3}  $h^*a\in F_\gen(Z, h^*D)$, for all $k$-morphisms 
$h:Z\to X$ with $Z$ smooth quasi-projective, and $\dim(Z)\le n$,  such that $|h^*D|$ is SNCD.
\end{enumerate}
Furthermore, if there is  a compactification $(\ol{X}, \ol{D}+B)$ of $(X,D)$, such that 
$\ol{X}$ is smooth projective and $|\ol{D}+B|$ is SNCD, the above conditions are equivalent to 
\begin{enumerate}[resume*]
\item\label{cor:cut-by-curve4} $a\in F_{\gen}(X,D)$.
\end{enumerate}
\end{cor}
\begin{proof}
We assume $D\neq \emptyset$ else there is nothing to prove.
The implication \ref{cor:cut-by-curve1}$\Rightarrow$ \ref{cor:cut-by-curve2} is direct.
The implication \ref{cor:cut-by-curve1}$\Rightarrow$ \ref{cor:cut-by-curve4} always holds,
hence so does \ref{cor:cut-by-curve2}$\Rightarrow$ \ref{cor:cut-by-curve3}.
Assume we have a compactification $(\ol{X}, \ol{D}+B)$ of $(X,D)$ with $\ol{X}$ smooth projective 
and $|\ol{D}+B|$ SNCD.
We may assume that $|\ol{D}|$ and $|B|$ have no common components.
For any $N\ge 1$ we have 
\ml{cor:cut-by-curve5}{F_{\gen}(X,D)\cap \tF(\ol{X}, N(\ol{D}+B))
=F_{\gen}(X,D)\cap F_{\gen}(\ol{X}, N(\ol{D}+B))\\
=F_\gen(\ol{X}, \ol{D}+ NB)=\tF(\ol{X}, \ol{D}+ NB)\subset \tF(X,D),}
where the first and the last equality hold by Theorem \ref{thm:ZNP} and the equality in the middle holds by
definition of $F_{\gen}$.
Since there exits an $N\ge 0$, such that $a\in \tF(\ol{X}, N(\ol{D}+B))$, the equality 
\eqref{cor:cut-by-curve5} proves the implication \ref{cor:cut-by-curve4}$\Rightarrow$ \ref{cor:cut-by-curve1} 
under the extra assumption on $(X,D)$.
It remains to show \ref{cor:cut-by-curve3}$\Rightarrow$ \ref{cor:cut-by-curve1} 
(assuming resolution of singularities in dimension $\le n$).
Let $(\ol{X}, \ol{D}+B)$ be any (not necessarily smooth) compactification of $(X,D)$.
Take $N\ge 0$, such that $a\in \tF(\ol{X}, N(\ol{D}+B))$.
We claim that $\ref{cor:cut-by-curve3}$ implies 
\[a\in \tF(\ol{X}, \ol{D}+NB)\subset \tF(X,D).\]
Indeed, by \cite[Corollary 4.18]{RS} and since $F$ has level $n$, it suffices  to show 
\eq{cor:cut-by-curve6}{\rho^*a\in \tF(\sO_L,\fm_L^{-v_L(\ol{D}+NB)})}
for any $\rho: \Spec L\to U$, where $L$ is  a henselian discrete valuation field of geometric type over $k$,
which has ${\rm trdeg}(L/k)\le n$, and where $v_L(\ol{D}+NH)$ denotes the multiplicity of $\rho^*(\ol{D}+NH)$ 
on $\Spec \sO_L$. (Note that $\rho$ uniquely extends to $\Spec \sO_L\to \ol{X}$.)
Let  $\ol{Z}_1$ be the closure of the image of $\Spec \sO_L$ in $\ol{X}$.
Using the Chow Lemma we find a proper morphism $\bar{h}: \ol{Z}\to\ol{X}$ which maps birational onto 
$\ol{Z}_1$ and such that $\ol{Z}$ is a projective $k$-scheme;
using resolutions of singularities in dimension $\le n$, we can additionally assume
$\ol{Z}$ is smooth  and $|\bar{h}^*(\ol{D}+B)|$ is SNCD. 
Note that by the valuative criterion for properness, $\Spec \sO_L\to \ol{X}$ factors via $\bar{h}$.
Set $Z:=\bar{h}^{-1}(X)$ and denote by $h: Z\to X$ the restriction of $\bar{h}$. 
By \ref{cor:cut-by-curve3} and the above we find
\[\bar{h}^*a\in F_{\gen}(Z, h^*D)\cap\tF(\ol{Z}, N\bar{h}^*(\ol{D}+B))=\tF(\ol{Z},\bar{h}^*(\ol{D}+NB)),\]
where the equality follows from \eqref{cor:cut-by-curve5} applied to $Z$ instead of $X$.
Pulling $\bar{h}^*a$ further back to $\Spec \sO_L$ yields \eqref{cor:cut-by-curve6}.
This completes the proof.
\end{proof}

\begin{exs}\label{exs:ZNP}
We list some examples where we can apply Corollary \ref{cor:cut-by-curve} unconditionally (recall this is the case if $\ch(k)=0$ or if $F$ has level$\leq 3$):
\begin{enumerate}[label=(\arabic*)]
\item Let $G$ be a commutative algebraic group over $k$. By \cite[Theorem 5.2]{RS}  
$G$ has level $1$. By \cite[Theorem 7.20]{RS} we obtain a 
cut-by-curves criterion for a conductor defined by Kato-Russell  (see \cite{Kato-Russell}).
%, which to our knowledge is new.
\item (${\rm char}(k)=p>0$) Let $F$ be one of the following two reciprocity sheaves
\[\Sm\ni X\mapsto \Hom_{\rm cts}(\pi_1^{\et}(X), \Q/Z)\quad \text{or}\quad
\Sm\ni X\mapsto \text{Lisse}^1_{\ell}(X),\]
where $\text{Lisse}^1_{\ell}(X)$ denotes the group of isomorphism classes of lisse $\ol{\Q}_{\ell}$-sheaves of rank 1 on $X$.
Then $F$ has level 1 and we obtain a cut-by-curves criterion for the Artin conductor of torsion characters or 
lisse rank 1 sheaves, by \cite[Theorem 8.8, Corollary 8.10]{RS}.
This also follows (by a different method) from the main results in \cite{KatoSwan} and \cite{Matsuda},
see \cite[Corollary 2.8]{Kerz-Saito}.
Also note that the statement for $D=\emptyset$ and $F=\Hom_{\rm cts}(\pi_1^{\et}(-), \Q/Z)$
is a direct consequence of the Zariski-Nagata purity theorem. 
\item (${\rm char}(k)=0$) By \cite[Theorem 6.11]{RS} assigning to $X\in \Sm$ the group of isomorphism classes 
of integrable rank 1 connections defines a reciprocity sheaf of level 1. We obtain a cut-by-curves criterion
for the irregularity. In the tame case (i.e. regular connections at infinity) this was proven by Deligne in 
\cite[II, Proposition 4.4]{Deligne-ED} - by a different method and even for higher rank connections.
In general this is  well-known, but we don't know a reference.
\item (${\rm char}(k)=p>0$)
 Let $G$ be a finite flat algebraic $k$-group ($G$ infinitesimal unipotent, like $\alpha_p$, is the most interesting case).
 By \cite[Theorem 9.12]{RS} the presheaf $X\mapsto H^1_{\rm fppf}(X, G)$ defines a reciprocity sheaf of level 
 at most $2$. Thus we get a cut-by-surfaces criterion for the ramification of $G$-torsors, which we believe to be  new.
 \item (${\rm char}(k)=p>0$) Let $\epsilon:\Sm_{\et}\to \Sm_{\Nis}$ be the change of sites map.
 We expect the reciprocity sheaves $R^n\epsilon_*(\Q/\Z(n-1))$ to be of level $n$, 
 where $\Q/\Z(n-1)$ denotes the \'etale motivic complex of weight $n-1$ with $\Q/\Z$-coefficients so that 
for $n=2$ (resp. $n=3$), 
we get a cut-by-surfaces (resp. cut-by-threefolds) criterion for the ramification of $R^n\epsilon_*(\Q/\Z(n-1))$. 
 We hope to come back to this point somewhere else.
\end{enumerate}
\end{exs}

\section{An application to rational singularities}

%\begin{para}\label{para:rat-sing}
In this section we assume ${\rm char}(k)=0$ and give an application of Corollary \ref{cor:cut-by-curve} to the theory of singularities. Recall that by Kempf's criterion a separated finite type $k$-scheme $X$ of pure dimension $d$ 
has {\rm rational singularities} if and only if it is normal, Cohen-Macaulay, and for one/any
resolution of singularities $f: Y\xr{\simeq} X$ we have $f_*\omega_{Y/k}=\omega_{X/k}$,
where $\omega_{Y/k}=\Omega^{d}_{Y/k}$  and $\omega_{X/k}= j_*\Omega^{d}_{X_{\rm sm}/k}$,
with $j: X_{\rm sm}=(\text{smooth locus})\inj X$, are the relative dualizing sheaves  over $k$.
There are various alternative descriptions of rational singularities, all relying on some sort of resolutions
(alterations, Macaulifications, etc...), see \cite{Kovacs} for the state  of the art.
Corollary \ref{cor:cut-by-curve} yields a resolution-free characterization as shown in the next theorem.
%\end{para}

\begin{thm}\label{cor:rat-sing}
Assume ${\rm char}(k)=0$.\footnote{The statement is extended to the case $\ch(k)>0$ in \cite[Cor.7.3]{RS-AS}. }
Let $X$ be an affine $k$-scheme of  finite type of pure dimension $d$, 
that is normal and Cohen-Macaulay. Let $D$ be an effective Cartier divisor on $X$, such that $(X,D)\in\uMCor$
(i.e., $D$ contains the singular locus of $X$). 
The following are equivalent:
\begin{enumerate}
\item $X$ has rational singularities.
\item We have  $\widetilde{\Omega^d}(X,D)=\Omega^d_{\gen}(X,D)$.
\end{enumerate}
\end{thm}
\begin{proof}
We find an open immersion $j:V\inj X$, such that 
$V\supset (X\setminus|D|)\cup X^{(1)}$
 and such that $D_V=j^*D$  has SNC support.
We obtain
\[\Omega^d_{\gen}(X,D)\stackrel{(1)}{=} \Omega^d_{\gen} (V,D_V)
\stackrel{(2)}{=} \widetilde{\Omega^d}(V,D_V)
\stackrel{(3)}{=}\Gamma(V,\Omega^d_V(\log D_V)\otimes_{\sO_V}\sO_V(D_V-D_{V,\red}))\]
\[\stackrel{(4)}{=} \Gamma(V,\Omega^d_V\otimes_{\sO_V} j^*\sO_X(D))
=\Gamma(X,j_*(\Omega^d_V\otimes_{\sO_V} j^*\sO_X(D)))
\stackrel{(5)}{=} \Gamma(X,\omega_{X/k}\otimes_{\sO_X} \sO_X(D)),\]
where $(1)$ holds by definition, $(2)$ follows from Corollary \ref{cor:cut-by-curve}  and resolutions of singularities, 
$(3)$ follows from \cite[Corollary 6.8]{RS}, $(4)$  holds since $d=\dim V$, 
and $(5)$ follows from the projection formula and a well-known formula for the dualizing 
sheaf $\omega_{X/k}$
of a normal Cohen-Macaulay $k$-scheme.  
On the other hand, if $f: Y\to X$ is a resolution of singularities such that $f^*D$ has SNC support,
then $(X,D)$ and $(Y, f^*D)$ are isomorphic in $\uMCor$ and we find similarly
\[\widetilde{\Omega^d}(X,D)= \widetilde{\Omega^d}(Y,f^*D)=
\Gamma(Y, \omega_{Y/k}\otimes_{\sO_Y} \sO_{Y}(f^*D)) =
\Gamma(X, f_*\omega_{Y/k}\otimes_{\sO_X}\sO_{X}(D)).\]
Thus 
\[\widetilde{\Omega^d}(X,D)=\Omega^d_{\gen}(X,D)\Longleftrightarrow 
\Gamma(X, f_*(\omega_{Y/k})\otimes_{\sO_X}\sO_{X}(D))= \Gamma(X,\omega_{X/k}\otimes_{\sO_X} \sO_X(D)).\]
Since $X$ is affine, the equality on the right is equivalent to 
\[f_*(\omega_{Y/k})\otimes_{\sO_X}\sO_{X}(D)=\omega_{X/k}\otimes_{\sO_X} \sO_X(D),\]
which is equivalent to $f_*(\omega_{Y/k})=\omega_{X/k}$.
Thus the statement follows from  Kempf's criterion stated in the beginning of this section.
%, see \ref{para:rat-sing}.
\end{proof}

%\bibliographystyle{amsalpha}
%\bibliography{Zariski-Nagta}

\providecommand{\bysame}{\leavevmode\hbox to3em{\hrulefill}\thinspace}
\providecommand{\MR}{\relax\ifhmode\unskip\space\fi MR }
\providecommand{\MRhref}[2]{%
  \href{http://www.ams.org/mathscinet-getitem?mr=#1}{#2}
}
\providecommand{\href}[2]{#2}

\end{document}